\documentclass{article}
\usepackage{latexsym,enumerate}
\usepackage{amsmath,amsthm,amsfonts,amsbsy,amsgen,amscd,amsxtra,amstext,amsopn,amssymb}
\usepackage[all]{xy}

\def\N{\mathbb{N}}
\def\Z{\mathbb{Z}}
\def\R{\mathbb{R}}

\def\C{\mathbb{C}} 
 
\def\A{\mathbb{A}}
\def\proj{\mathbb{P}}

\renewcommand{\a}{\alpha}

\renewcommand{\b}{\beta}
\renewcommand{\d}{\delta}

\newcommand{\eps}{\varepsilon}
\newcommand{\g}{\gamma}
\newcommand{\G}{\Gamma}
\newcommand{\la}{\lambda}

\newcommand{\var}{\varphi}

\renewcommand{\th}{\vartheta}

\newcommand{\Om}{\Omega}
\newcommand{\om}{\omega}

\def\mZ{{\mathcal Z}}

\renewcommand{\hat}{\widehat}

\renewcommand{\tilde}{\widetilde}

\def\Oh{{\cal O}}

\def\algorithm{\begin{center}
               \begin{minipage}{6in}
               \begin{tabbing}
               \marks}
\def\falgorithm{\end{tabbing}
                \end{minipage}
                \end{center}}
\def\marks{nn\= nn\= nn\= nn\= nn\= nn\= nn\= \kill}

\def\PSPACE{{\sf PSPACE}}

\def\PR{{\rm P}_{\kern-1pt\R}}

\def\PC{{\rm P}_{\kern-1pt\C}}
\def\NPR{{\rm NP}_{\kern-1pt\R}}
\def\NPC{{\rm NP}_{\kern-2pt\C}}
\def\DNPR{{\rm DNP}_{\kern-1pt\R}}
\def\DNPC{{\rm DNP}_{\kern-2pt\C}}
\def\PAR{{\rm PAR}_{\kern-1pt\R}}
\def\PHR{{\rm PH}_{\kern-1pt\R}}
\def\DPHR{{\rm DPH}_{\kern-1pt\R}}

\def\FPR{{\rm FP}_{\kern-1pt\R}}
\def\FPC{{\rm FP}_{\kern-1pt\C}}
\def\FPAR{{\rm FPAR}_{\kern-0.4pt\R}}
\def\FPARC{{\rm FPAR}_{\kern-0.4pt\C}}

\def\CPRi{{\rm \#P}_{\kern-2pt\R}}
\def\CPCi{{\rm \#P}_{\kern-2pt\C}}

\def\FEASR{{\mbox{\sc Feas}_{\kern-0.5pt\R}}}  
\def\FEASRbit{{\mbox{\sc Feas}^0_{\kern-1pt\R}}}
\def\SAS{{\mbox{\sc SAS}_{\kern-0.5pt\R}}}
\def\SASbit{{\mbox{\sc SAS}_{\kern-1pt\R}^0}}
\def\HNC{{\mbox{\sc HN}_{\kern-1pt\C}}}
\def\HNCbit{{\mbox{\sc HN}^0_{\kern-1pt\C}}}
\def\QASC{{\mbox{\sc QAS}_{\kern-1pt\C}}}

\def\DIMR{{\mbox{\sc Dim}_{\kern-0.5pt\R}}}
\def\DIMC{{\mbox{\sc Dim}_{\kern-0.5pt\C}}}
\def\DIMadd{{\mbox{\sc Dim}_{\kern-0.5pt\add}}}
\def\DIMRbit{{\mbox{\sc Dim}^0_{\kern-0.5pt\R}}}
\def\DIMCbit{{\mbox{\sc Dim}^0_{\kern-0.5pt\C}}}
\def\REACH{{\mbox{\sc Reach}_{\kern-0.5pt\R}}}
\def\REACHbit{{\mbox{\sc Reach}^0_{\kern-0.5pt\R}}}
\def\CREACHbit{{\mbox{\sc CReach}^0_{\kern-0.5pt\R}}}

\def\REACH{{\mbox{\sc Reach}_{\kern-0.5pt\R}}}
\def\REACHbit{{\mbox{\sc Reach}^0_{\kern-0.5pt\R}}}
\def\CREACHbit{{\mbox{\sc CReach}^0_{\kern-0.5pt\R}}}

\newcommand{\HHNC}{\mbox{\sc HN}_{\kern-1pt\C}}

\newcommand{\partder}[2]{\frac{\partial #1}{\partial #2}}
\newcommand{\ord}[1]{\mathrm{ord}_{#1}}

\def\mF{{\mathcal F}}

\newcommand{\ud}{\mathrm{d}}

\def\im{\mathrm{im}\,}
\def\id{\mathrm{id}}

\def\FPk{{\rm FP}_{\kern-1pt k}}
\def\Pk{{\rm P}_{\kern-1pt k}}
\def\NPk{{\rm NP}_{\kern-2pt k}}

\newcommand\CPpar[1]{{\rm \#P}_{\kern-2pt #1}}

\def\HNk{\textsc{HN}_{\kern-1pt k}}
\def\DIMk{{\mbox{\sc Dim}_{\kern-0.5pt k}}}

\newcommand{\comment}[1]{}

\def\H{\mathbb{H}}

\def\Hdr{H_{\rm{dR}}}
\def\tot{{\rm tot}}
\def\mU{{\mathcal U}}
\def\Res{{\rm Res}}

\theoremstyle{plain}
\newtheorem{lemma}{Lemma}[section]	
\newtheorem{theorem}[lemma]{Theorem}
\newtheorem{proposition}[lemma]{Proposition}
\newtheorem{corollary}[lemma]{Corollary}

\theoremstyle{definition}

\theoremstyle{remark}
\newtheorem{remark}[lemma]{Remark}
\newtheorem{example}[lemma]{Example}

\title{Effective de Rham Cohomology -- The General Case}
\author{Peter Scheiblechner\\
Lucerne University of Applied Sciences and Arts\\
School of Engineering and Architecture\\
6048 Horw, Switzerland  \\
peter.scheiblechner@hslu.ch
}

\date{}

\begin{document}

\maketitle

\begin{abstract}
Grothendieck has proved that each class in the de Rham cohomology of a smooth complex affine
variety can be represented by a differential form with polynomial coefficients.
We prove a single exponential bound on the degrees of these polynomials for 
varieties of arbitrary dimension. More precisely, we show that the
$p$-th de Rham cohomology of a smooth affine variety of dimension $m$ and  degree $D$
can be represented by differential forms of degree $(pD)^{\Oh(pm)}$.
This result is relevant for the algorithmic computation of the cohomology, but
is also motivated by questions in the theory of ordinary differential equations
related to the infinitesimal Hilbert 16th problem.
\end{abstract}


\bigskip



\section{Introduction}
Let $X$ be a smooth variety in $\C^n$.
A fundamental result of Grothendieck says that the cohomology of $X$ can be
described in terms of \textit{algebraic} differential forms on~$X$~\cite{gro:66}.
More precisely, he proved that the singular cohomology of $X$ is isomorphic to
the algebraic de Rham cohomology $\Hdr^\bullet(X)$, which is defined as the
cohomology of the complex of algebraic differential forms on $X$.
Hence, each cohomology class in $\Hdr^p(X)$ can be represented by a $p$-form
\begin{equation}\label{eq:diffForm}
\omega=\sum_{i_1<\cdots<i_p}\omega_{i_1\cdots i_p}\ud X_{i_1}\wedge\cdots\wedge \ud X_{i_p},
\end{equation}
where the $\omega_{i_1\cdots i_p}$ are polynomial functions on $X$. However, Grothendieck's proof is
not {\em effective}, i.e., it gives no information about the degrees of the polynomials~$\omega_{i_1\cdots i_p}$,
say in terms of defining equations for $X$.
In~\cite{sch:12} we proved a single exponential bound on their degrees in the case that $X$ is
a hypersurface. In particular, if $X\subseteq\C^n$ is a smooth hypersurface of degree $D$,
then each cohomology class in $\Hdr^p(X)$ can be represented by a differential form $\om$ as
in~\eqref{eq:diffForm},
where the~$\omega_{i_1\cdots i_p}$ are polynomials of degree $D^{\Oh(pn)}$.
The present paper generalizes this result to smooth varieties of arbitrary dimension.

\subsection{Motivation}
It is a long standing open question in algorithmic real algebraic geometry to
find a single exponential time algorithm for computing the Betti numbers of a
semialgebraic set. Single exponential time algorithms are known, e.g., for
counting the connected components and computing the Euler characteristic of a
semialgebraic set (for an overview see~\cite{basu:08}, for details and
exhaustive bibliography see~\cite{bpr:03}). The best result in this direction
states that for fixed $\ell$ one can compute the first $\ell$ Betti numbers of
a semialgebraic set in single exponential time~\cite{basu:06}.

Over the complex numbers, one approach for computing Betti numbers is to 
compute the algebraic de Rham cohomology. In~\cite{ota:98,wal:00} the de Rham
cohomology of the complement of a complex affine variety is computed using
Gr\"{o}bner bases for $\mathcal{D}$-modules. This algorithm is extended in
\cite{wal:00b} to compute the cohomology of a projective variety.
However, the complexity of these algorithms is not analyzed, and due to their
use of Gr\"{o}bner bases a good worst-case complexity is not to be expected.
In~\cite{bus:09} a single exponential time (in fact, parallel polynomial
time) algorithm is given for counting the connected components, i.e., computing
the zeroth de Rham cohomology, of a (possibly singular) complex variety.
This algorithm is extended in~\cite{sch:10} to one with the same asymptotic complexity
for computing equations for the components.
The first single exponential time algorithm for computing \textit{all} Betti
numbers of an interesting class of varieties is given in~\cite{sch:09}. Namely,
this paper shows how to compute the de Rham cohomology of a smooth projective
variety in parallel polynomial time. In terms of structural complexity in the Turing model, these
results are the best one can hope for, since the problem of computing a fixed
Betti number (e.g., deciding connectedness) of a complex affine or projective
variety defined over the integers is $\PSPACE$-hard~\cite{sch:07}.

Besides being relevant for algorithms, our question also has connections to the
theory of ordinary differential equations. 
The long standing infinitesimal Hilbert 16th problem has been solved
in~\cite{bny:10}.
The authors derive a bound on the number of limit cycles generated from
nonsingular energy level ovals (isolated periodic trajectories) in a
non-conservative perturbation of a Hamiltonian polynomial vector field in the
plane.
It seems that their proof can be considerably generalized to solutions of
certain linear systems of Pfaffian differential equations. Examples of such
systems are provided by period matrices of polynomial maps, once the
corresponding Gauss-Manin connexion can be explicitly constructed. For this
construction one needs degree bounds for generators of the cohomology of the
generic fibers of the polynomial map.

\subsection{Known Cases}\label{ss:knownCases}
We have shown in~\cite[Theorem~3.3]{bus:09} that the zeroth de Rham cohomology of $X$ is
isomorphic to its
zeroth singular cohomology even when $X$ is singular, and that
this cohomology has a basis of degree $d^{\Oh(n^2)}$, if $X$ is defined by polynomials
of degree $\le d$.

It follows from the results of~\cite{sch:09} that if $X$ has no singularities
at infinity, i.e., the projective closure of $X$ in $\proj^n$ is smooth, then
each class in $\Hdr^\bullet(X)$ can be represented by a differential form of
degree at most $m(em+1)D$, where $m=\dim X$, $D=\deg X$, and $e$ is the maximal
codimension of the irreducible components of $X$. However, in general $X$ does
have singularities at infinity, and resolution of singularities has a very bad
worst-case complexity~\cite{bgmw:11}.

Another special case with known degree bounds is the complement of a projective hypersurface,
which we will actually use in this paper (see the proof of Theorem~\ref{thm:effDeRhamHyper}). The
statement follows essentially from~\cite{ded:90} and~\cite{dimc:90}, the argument can be found
in~\cite[Corollary 6.1.32]{dimc:92}, see also~\cite{dimc:91}. Let $f\in\C[X_0,\ldots,X_n]$ be a
homogeneous polynomial, and consider $U:=\proj^n\setminus\mZ(f)$, which is an affine variety.
Then, each class in $\Hdr^p(U)$ is represented by a (homogeneous) differential form
$$
\a/f^p\quad\text{with}\quad\deg \a=p\deg f
$$
(see \S\ref{ss:basics} for the definition of the degree of a differential form).
Since this result was already proved by Griffiths in a special case~\cite{gri:69}, we call it the
Griffiths-Deligne-Dimca (GDD) bound.

\subsection{Main Result}
The main result of this paper is that the algebraic de Rham cohomology of a smooth affine variety
can be represented by differential forms of single exponential degree. More precisely, we prove
the following.
\begin{theorem}\label{thm:main}
Let $k$ be an algebraically closed field of characteristic zero, and let $X$ be a smooth affine
$m$-dimensional variety of degree $D$ over $k$. Then each cohomology class in~$\Hdr^p(X)$
can be represented by a differential $p$-form of degree at most
$$
2^{2p m+6m+2}p^{2p m+6m+1} D^{4p m+10m+1}+D^{m+1}=(p D)^{\Oh(p m)}.
$$
\end{theorem}
\begin{remark}
\begin{enumerate}[(i)]
\item Note that we kept the additive term $D^{m+1}$, so that the bound is valid
for $p=0$ as well (cf.\ Proposition~\ref{prop:boundConnComp1}).
\item If $X\subseteq\A^n$, then the term $D^{m+1}$ can be replaced by $(n+1)D^2/4$
(Corollary~\ref{cor:redIrred}).
\item For a hypersurface $X$, the slightly better bound of $D^{\Oh(pn)}$ is proved in~\cite{sch:12}.
\end{enumerate}
\end{remark}

\subsection{Proof Ideas}
The proof of Theorem~\ref{thm:main} consists of two major steps. In a first step we reduce the
question to the case of an irreducible locally closed hypersurface, and in a second step we prove
a bound for this special case (Theorem~\ref{thm:effDeRhamHyper}).

The reduction to irreducible~$X$ (Corollary~\ref{cor:redIrred}) is an easy consequence of our
characterization of the zeroth de Rham cohomology and the bound from~\cite{bus:09}. 
In the present paper, using newer versions of the effective Nullstellensatz due
to Koll\'ar~\cite{kol:98} and Jelonek~\cite{jel:05}, we improve this bound in two directions
(Propositions~\ref{prop:boundConnComp1} and~\ref{prop:boundConnComp2}). We also give an example
showing that the second version is sharp up to a factor of $n+1$, where $n$ is the ambient
dimension.

Then, if~$X$ is irreducible
of dimension $m$, a generic projection to a linear subspace of dimension $m+1$ restricts to a
birational map from~$X$ to a hypersurface $Y$. By Zariski's Main Theorem, there exist open dense
subsets $U\subseteq X$ and $V\subseteq Y$, such that the projection maps $U$ isomorphically onto
$V$. Using a {\em geometric resolution} of~$X$~\cite{ghmp:95,par:95}, we construct these locally
closed sets effectively. More precisely, there exists a polynomial $f$ of degree $\le D$ in $m+1$
variables, such that $U=X\setminus\mZ(g)$ and $V=\mZ(f)\setminus\mZ(g)$, where $g$ is a partial
derivative of~$f$ (Lemma~\ref{lem:biratHyper}).
It follows that a degree bound for $\Hdr^\bullet(V)$ implies a degree bound for $\Hdr^\bullet(U)$
(Corollary~\ref{cor:transHyperPatch}). Moreover, it is also not difficult to see that one can
cover $X$ by such principal open subsets~$U_i$ (Corollary~\ref{cor:biratHyperCover}). To finish
the first major step, it remains to show how degree bounds for the open patches yield a bound for
$\Hdr^\bullet(X)$ (Theorem~\ref{thm:effHyperCoh}). This is done using the concept of hypercohomology,
which through a spectral sequence immediately yields an effective description of $\Hdr^\bullet(X)$
(Lemma~\ref{lem:hyperCoh}) in terms of the total complex of the \v{C}ech-de Rham double complex.
In order to derive a bound for the usual description in terms of global sections, we make the
\v{C}ech cohomology effective (Lemma~\ref{lem:chechCoh}). In this construction we make essential
use of Jelonek's effective Nullstellensatz~\cite{jel:05}.

The main idea for the treatment of the locally closed hypersurface $V$ is the same as
in~\cite{sch:12}, namely to prove an
effective version of the Gysin sequence (Theorem~\ref{thm:gysinSequ}) and use the GDD bound
for complements of hypersurfaces mentioned in \S\ref{ss:knownCases}.  However, in~\cite{sch:12}
we considered the case of a closed hypersurface $X\subseteq\A^n $, in which case the Gysin
sequence yields an isomorphism $\Hdr^p(\A^n\setminus X)\stackrel{\sim}{\to}\Hdr^{p-1}(X)$, which we
made effective. Here, we have only a locally closed subset $V=\mZ(f)\setminus\mZ(g)$ of~$\A^n$.
One
idea is to treat $V$ as a closed subset of $\A^n\setminus \mZ(g)$. Though we are able to construct
an effective residue map in this setting, this map may not be surjective, since the surrounding
space $\A^n\setminus \mZ(g)$ has non-trivial cohomology in general. Our solution is to consider
$V$ as a codimension 2 complete intersection $Z\subseteq\A^{n+1}$ through the isomorphism
$\A^n\setminus \mZ(g)\simeq\mZ(gY-1)\subseteq\A^{n+1}$. Since the Gysin sequence in the version
of~\cite{har:67} does not hinge on codimension 1, we get an isomorphism
$\Hdr^{p+1}(W)\stackrel{\sim}{\to}\Hdr^{p-2}(Z)$, where $W:=\A^{n+1}\setminus Z$
(Corollary~\ref{cor:ResIso}). Luckily, while the general complete intersection case seems
considerably more difficult, we are able to prove the crucial Lemma (Lemma~\ref{lem:groth})
exactly in this special case and obtain an effective Gysin sequence (Theorem~\ref{thm:effGysin}).

However, the price to pay is that now the complement $W$ is not affine anymore, so we need
hypercohomology also to realize the de Rham cohomology of~$W$, and a spectral sequence argument
(Lemma~\ref{lem:hyperCoh}) to apply the GDD bound (Lemma~\ref{lem:boundHyperComplement}). Using
local cohomology and another spectral sequence argument, we get a more succinct description of
$\Hdr^\bullet(W)$ in terms of sheaf cohomology (Lemma~\ref{lem:deRhamLocal}). 

\section{Preliminaries}
\subsection{Basic Notations and Facts}
\label{ss:basics}
Throughout this paper, let $k$ be an algebraically closed field of characteristic zero, and
$R:=k[X_1,\ldots,X_n]$. 
An {\em algebraic set} or {\em closed variety} is the common zero set of polynomials
$f_1,\ldots,f_r\in R$ in the affine space $\A^n=k^n$, i.e.,
$$
X=\mZ(f_1,\ldots,f_r)=\{x\in k^n\,|\,f_1(x)=\cdots=f_r(x)=0\}.
$$
Note that $X$ may be reducible. More generally, the term {\em variety} will refer to a
{\em locally closed set}, i.e., a Zariski open subset $X\subseteq Y$ of a closed variety $Y\subseteq\A^n$.
A variety $X$ is called {\em affine} iff it is a {\em principal open subset} of $Y$, i.e.,
$$
X=Y\setminus\mZ(g)=\mZ(f_1,\ldots,f_r)\setminus\mZ(g),\quad f_1,\ldots,f_r,g\in R.
$$
Indeed, these are exactly the varieties which are isomorphic to a closed variety, namely
$\mZ(f_1,\ldots,f_r,X_{n+1}g-1)\subseteq\A^{n+1}$. The (vanishing) ideal of an affine variety~$X$ is
defined as $I(X)=\{f\in R\,|\, f(x)=0\ \forall x\in X\}$. By Hilbert's Nullstellensatz, $I(X)$ is
the radical of the ideal $(f_1,\ldots,f_r):(g)$.  The {\em coordinate ring} $A=k[X]$ of $X$ is the
localization $(R/I(X))_g$. Such a ring $A$ will be called a (reduced) {\em affine} $k$-algebra.

In general, a variety is not affine. However, if $X=Y\setminus\mZ(g_1,\ldots,g_r)$ with affine $Y$,
then $X$ is covered by the affine open subsets $U_i:=X\setminus\mZ(g_i)=Y\setminus\mZ(g_i)$,
$1\le i\le r$.

The {\em dimension} $\dim X$ is the {\em Krull dimension} of $X\subseteq\A^n$ in the Zariski
topology. In the case $\dim X=n-1$ we call $X$ a {\em hypersurface}. The {\em degree} $\deg X$ of
an irreducible variety $X\subseteq\A^n$ of dimension~$m$ is defined as the maximal cardinality of
$X\cap L$ over all affine linear subspaces $L\subseteq\A^n$ of dimension $n-m$~\cite[\S5A]{mum:76}.
We define the (cumulative) degree $\deg X$ of a reducible variety~$X$ to be the sum of the degrees
of {\em all} irreducible components of $X$. It follows essentially from B\'{e}zout's Theorem that
if $X\subseteq\A^n$ is closed and defined by polynomials of degree $\le d$, then $\deg X\le d^n$~\cite{cgh:91}.

An important ingredient of our work is the effective Nullstellensatz, which
was first proved in~\cite{brow:87,koll:88}. A recent version of Jelonek~\cite{jel:05}
generalizes the statement to polynomials on general affine varieties.

\begin{theorem}\label{thm:effNS}
Let $X\subseteq\A^n$ be a closed subvariety of dimension $m$ and degree~$D$, and
let $g_1,\ldots,g_t\in R$ be polynomials of degree at most $d\ge 1$ without
common zeros in $X$.
Then there exist polynomials $h_i\in R$ with $1\equiv \sum_i h_i g_i\mod{I(X)}$ and
$$
\deg(h_ig_i)\le \left\{
\begin{array}{ll}
Dd^t & \text{if}\quad t\le m, \\
Dd^m & \text{if}\quad t>m,\ d\ge 3,\ m\ge n-1, \\
2Dd^m-1 & \text{else}. \\
\end{array}
\right.
$$
\end{theorem}
\begin{proof}
The first and the last case is due to~\cite{jel:05}. The slightly better bound
in the middle case follows from~\cite{koll:88}.
\end{proof}

There is another version of the effective Nullstellensatz for arbitrary ideals due to
Koll\'ar~\cite{kol:98}. We note a variant which follows easily from Theorem 6.2 in that paper.
\begin{theorem}\label{thm:effNSIdeals}
Let $X_1,\ldots,X_t\subseteq\A^n$ be closed varieties with $X_1\cap\cdots\cap X_t=\emptyset$.
Then there exist $f_i\in I(X_i)$ such that
$$
\sum_i f_i=1\quad\text{and}\quad \deg f_i\le (n+1)\prod_i\deg X_i.
$$
\end{theorem}

\subsection{Completions}
\label{ss:completions}
In our proof we will also use the process
of completion~\cite[Chapter 7]{eise:95}. Let $A$ be an affine algebra and $I$
an ideal in $A$. The \textit{completion} $\hat{A}=\hat{A}_I$ of~$A$ with
respect to $I$ is defined as the inverse limit of the factor rings
$A/I^\nu$, $\nu\ge 0$. There is a canonical map $A\to\hat{A}$, whose kernel is
$\bigcap_\nu I^\nu$, thus it is injective in our case.
Alternatively, if $I=(f_1,\ldots,f_r)$, one can define $\hat{A}$ as
$A[[T_1,\ldots,T_r]]/(T_1-f_1,\ldots,T_r-f_r)$, so its elements
are power series in $f_1,\ldots,f_r$~\cite[Exercise 7.11]{eise:95}.
For instance, if $A=B[T]$ and $I=(T)$, then $\hat{A}=B[[T]]$ is the ring of
formal power series in $T$ with coefficients in $B$.

\subsection{Coherent Sheaves and Differential Forms}
\label{ss:cohSheaves}
Let $X$ be an affine variety. Then every $k[X]$-module~$M$ gives rise to a sheaf $\tilde{M}$ on
$X$ such that, on a principal open subset $U=X\setminus\mZ(g)$, the sections of~$\tilde{M}$ are
given by $\G(U,\tilde{M})=M_{g}$, the localization of $M$ at~$g$~\cite[Proposition II.5.1]{hart:77}.
Moreover, $\G(U,\tilde{M})$ is a $k[X]_g$-module which is compatible with restrictions. A sheaf
$\mF$ on $X$ is called {\em coherent} iff $\mF=\tilde{M}$ with a finitely generated $k[X]$-module~$M$.
An important example is the structure sheaf $\Oh_X=\tilde{k[X]}$.

More generally, a sheaf $\mF$ on a locally closed set~$X$ is called {\em coherent} iff~$X$ can be
covered by affine open subsets $U_i$ such that all the restrictions $\mF|U_i$ are coherent. In
particular, if 
$\mF$ is a coherent sheaf on $Y$ and~$X\subseteq Y$ is an open subset, then $\mF|X$ is a coherent
sheaf on $X$.

Now let $A$ be a $k$-algebra (commutative, with 1). The module of \textit{K\"{a}hler differentials}
$\Om_A:=\Om_{A/k}$ is defined as the $A$-module generated by symbols~$\ud f$ for all $f\in A$,
modulo the relations of Leibniz' rule and $k$-linearity of the \textit{universal derivation}
$\ud\colon A\to\Om_A$. For instance, for the polynomial ring $R$,
the module $\Om_R$ is free with basis $\ud X_1,\ldots,\ud X_n$, and the universal derivation is
given by $\ud f=\sum_i\partder{f}{X_i} \ud X_i$.
Now let $\Om_A^p:=\bigwedge^p\Om_A$ be the $p$-th exterior power of the $A$-module $\Om_A$. We
define the {\em exterior differential} $\ud\colon \Om^p_A\to\Om^{p+1}_A$ by setting
$\ud (f \ud g_1\wedge\cdots\wedge\ud g_p):=\ud f\wedge g_1\wedge\cdots\wedge\ud g_p$ for all
$f,g_i\in A$. It is easy to check that~$\ud$ satisfies the graded Leibniz' rule and
$\ud\circ\ud=0$. This way we obtain the \textit{de Rham complex}
$$
\Om^\bullet_A\colon A=\Om^0_A\stackrel{\ud}{\longrightarrow}\Om_A^1\stackrel{\ud}
{\longrightarrow}\cdots\stackrel{\ud}{\longrightarrow}\Om_A^p\longrightarrow \cdots
$$
of the algebra $A$.

Now let $Y$ be an affine variety with coordinate ring $A=k[Y]$. Then the {\em sheaf of regular
differential $p$-forms} on $Y$ is defined as $\Om_Y^p:=\tilde{\Om}_A^p$. The exterior differentials
glue together to maps of sheaves $\ud\colon\Om_Y^p\to\Om_Y^{p+1}$. More generally, if $X\subseteq Y$
is an open subset of an affine variety $Y$, then $\Om_X^p:=\Om_Y^p|X$. 
This way we obtain the \textit{(algebraic) de Rham complex}
$$
\Om^\bullet_X\colon \Oh_X=\Om^0_X\stackrel{\ud}{\longrightarrow}\Om_X^1\stackrel{\ud}
{\longrightarrow}\cdots\stackrel{\ud}{\longrightarrow}\Om_X^p\longrightarrow \cdots
$$
of the variety $X$.
We say that $X$ is {\em smooth} at $x\in X$ iff the stalk $\Om_{X,x}$ is a free $\Oh_{X,x}$-module
of rank $m$, where $m$ is the maximal dimension of all irreducible components of $X$ through $x$.
In fact, in this case there is only one such component. We say that $X$ is {\em smooth} iff it is
smooth at all its points.
Note that in this case $\Om_X^p=0$ for $p>m$.

The module of K{\"a}hler differentials of a complete ring may not be finitely generated (see,
e.g.,~\cite[Exercise 16.14]{eise:95}). In these cases we use the \textit{universally finite} module
of differentials, which is always finitely generated (see~\cite[\S11--12]{kun:86}). Let $A$ be an
affine algebra, $I$ an ideal in $A$, and $\hat{A}$ the completion of $A$ with respect to $I$. The
\textit{completion} of~$\Om_A$ with respect to $I$ is $\hat{\Om}_{\hat{A}}=\hat{A}\otimes_A\Om_A$
and is called the \textit{universally finite module of differentials of}~$\hat{A}$. There is a
\textit{universally finite derivation} $\ud\colon\hat{A}\to \hat{\Om}_{\hat{A}}$ which is
continuous, i.e., it commutes with infinite sums. For instance, for an affine algebra $B$ we have
$$
\hat{\Om}_{B[[T]]}=B[[T]]\otimes_{B[[T]]}\Om_{B[T]}=B[[T]]\ud T\oplus\Om_B,
$$
and the universally finite derivation is given by $\ud f=\partder{f}{T}\ud T+\ud_B f$ for
$f\in B[[T]]$, where $\partder{f}{T}$ denotes the formal partial derivative with respect to $T$,
and $\ud_B f$ is coefficient-wise application of the differential of $B$~\cite[Example 12.7]{kun:86}.

\subsection{Sheaf and Local Cohomology}
\label{ss:sheafCohomology}
Let $X$ be a variety. Formally, the $q$-th {\em sheaf cohomology} functor $H^q(X,\cdot)$ from the
category of sheaves of abelian groups on $X$ to the category of abelian groups is defined as the
$q$-th right derived functor of the global section functor $\G(X,\cdot)$,
$$
H^q(X,\cdot)=R^q\G(X,\cdot).
$$

Since this definition is computationally quite inconvenient, we prefer a different description.
Let $\mF$ be a coherent sheaf and $\mU:=\{U_i\,|\,0\le i\le t\}$ an affine open cover of $X$.
The sheaf cohomology of $\mF$ can be computed as the \v{C}ech cohomology with respect to $\mU$,
which is defined as follows. For a set of indices $0\le i_0,\ldots,i_q\le t$ denote
$U_{i_0\cdots i_q}:=U_{i_0}\cap\cdots\cap U_{i_q}$. We define the vector spaces
$$
C^q:=C^q(\mU,\mF):=\bigoplus_{0\le i_0<\cdots<i_q\le t}\mF(U_{i_0\cdots i_q})
$$
and the linear maps $\d^q\colon C^q\longrightarrow C^{q+1}$,
\begin{equation}\label{eq:cechDiff}
\d^q(\om)_{i_0\cdots i_{q+1}}:=\sum_{\nu=0}^{q+1}(-1)^\nu\om_{i_0\cdots\hat{i_\nu}\cdots i_{q+1}}|U_{i_0\cdots i_{q+1}}.
\end{equation}
Then one checks that $\d^{q+1}\circ\d^q=0$ for all $q\ge 0$, so that $(C^\bullet(\mU,\mF),\d^\bullet)$
is a complex called the {\em \v{C}ech complex}. Its cohomology defines the cohomology
$H^\bullet(\mU,\mF)$, i.e.,
$$
H^q(\mU,\mF):=\frac{\ker\d^q}{\im\d^{q-1}},\quad q\ge 0.
$$
By~\cite[Theorem III.4.5]{hart:77} there is a natural isomorphism
$$
H^q(X,\mF)\simeq H^q(\mU,\mF)\quad\text{for all}\quad q\ge 0.
$$
Generally, $H^0(X,\mF)$ is isomorphic to the space of global sections $\G(X,\mF)$. Moreover, if
$X$ is affine, the higher cohomology of $\mF$ vanishes, i.e.,
$$
H^q(X,\mF)=0\quad\text{for}\quad q>0.
$$ 

A relative variant of sheaf cohomology is local cohomology. Let $X$ be a variety and $Y\subseteq X$
a closed subset of $X$. For a sheaf of abelian groups $\mF$ on~$X$ let $\G_Y(X,\mF)$ be the
subgroup of $\G(X,\mF)$ consisting of all sections $s$ whose support is in $Y$, i.e.,
$$
\{s_x\ne 0\}\subseteq Y,
$$
where $s_x$ denotes the germ of $s$ in the stalk $\mF_x$. The $q$-th {\em local cohomology
functor with supports in} $Y$ from the category of sheaves of abelian groups on $X$ to the
category of abelian groups is defined as the $q$-th right derived functor of the functor
$\G_Y(X,\cdot)$,
$$
H_Y^q(X,\cdot)=R^q\G_Y(X,\cdot).
$$
Note that, if $X$ is irreducible and $Y\ne X$, then $\G_Y(X,\mF)=0$ for all coherent sheaves $\mF$
on $X$. Nevertheless, the definition gives non-trivial local cohomology groups also in this
case! There is an explicit description of local cohomology in terms of Koszul cohomology, which
can also be interpreted as \v{C}ech cohomology~\cite[Theorem 2.3]{har:67}. We will use this
interpretation in a special case in the proof of Lemma~\ref{lem:deRhamLocal}, 
see~\eqref{eq:isoSheafLocalCoh}. What makes local cohomology
particularly useful, is the following long exact sequence~\cite[Corollary 1.9]{har:67}.
Let $Y\subseteq X$ be closed and $U:=X\setminus Y$. Then there is an exact sequence
$$
\cdots\to H^{q-1}(X,\mF)\to H^{q-1}(U,\mF|U)\to H_Y^{q}(X,\mF)\to H^{q}(X,\mF)\to\cdots.
$$

\subsection{Hypercohomology and de Rham Cohomology}
\label{ss:hyperCohomology}

Let $X$ be a variety and consider a complex of coherent sheaves $(\mF^\bullet,\ud)$ on~$X$ with
$\mF^p=0$ for $p<0$. Then, for an affine open cover $\mU$, the \v{C}ech complexes $C^\bullet(\mU,\mF^p)$
as defined in~\S\ref{ss:sheafCohomology} fit together to the {\em \v{C}ech double complex}
$C^{\bullet,\bullet}:=C^{\bullet,\bullet}(\mU,\mF^{\bullet})$ by setting
$$
C^{p,q}(\mU,\mF^{\bullet})=\bigoplus_{i_0<\cdots<i_q} \mF^p(U_{i_0\cdots i_q})
\quad\text{for all}\quad p,q\ge 0.
$$
The two differentials are the one induced by the differential $\ud$ of $\mF$
and the \v{C}ech differential $\d^\bullet$ defined by~\eqref{eq:cechDiff}.
Define the total complex of $C^{\bullet,\bullet}$ by
$$
\tot^\ell(C^{\bullet,\bullet}):=\bigoplus_{p+q=\ell} C^{p,q},\quad
\ud^\tot:=\d^q+(-1)^q\ud\ \text{on}\ C^{p,q}.
$$
Since the two differentials commute, one easily checks that $\ud^\tot\circ\ud^\tot=0$. The
{\em hypercohomology} of the complex of sheaves $\mF^\bullet$ is defined as
the cohomology of the total complex
$$
\H^\ell(X,\mF^\bullet):=H^\ell(\tot^\bullet(C^{\bullet,\bullet}))\quad\text{for}\quad \ell\ge 0.
$$
As for any double complex~\cite[\S2.4]{mcc:85}, there are two spectral sequences
\begin{align*}
{}_{I}E_2^{p,q}=H^p(H^q(X,\mF^\bullet),\ud) & \quad\Rightarrow\quad \H^{p+q}(X,\mF^\bullet)\quad\mathrm{and}\\
{}_{II}E_2^{p,q}=H^q(H^p(C^{\bullet,\bullet},\ud),\d^\bullet) & \quad\Rightarrow\quad \H^{p+q}(X,\mF^\bullet).
\end{align*}
If $X$ is affine, then $H^q(X,\mF^\bullet)=0$ for all $q>0$. Consequently, the first spectral
sequence implies that $\H^\bullet(X,\mF^\bullet)\simeq H^\bullet(\G(X,\mF^\bullet))$. 

The {\em algebraic de Rham cohomology} of a variety $X$ is defined as the hypercohomology
of the algebraic de Rham complex
$$
\Hdr^\bullet(X):=\H^\bullet(X,\Om_X^\bullet).
$$
The corresponding double complex is called {\em \v{C}ech-de Rham double complex}.
In particular, if $X$ is affine, then we have
$$
\Hdr^\bullet(X)=H^\bullet(\Om_A^\bullet),\quad\text{where}\quad A=k[X].
$$
In general, since the $U_{i_0\cdots i_q}$ are affine, the last equation implies that the first term of the second spectral
sequence of the \v{C}ech-de Rham double complex is
$$
{}_{II}E_1^{p,q}=\bigoplus_{i_0<\cdots <i_q} \Hdr^p(U_{i_0\cdots i_q}) \quad\Rightarrow\quad \Hdr^{p+q}(X),
$$
even when $X$ is not affine.

Fundamental for us is the result of~\cite{gro:66} stating that if $k=\C$, then the de Rham
cohomology $\Hdr^\bullet(X)$ of a smooth variety $X$ is naturally isomorphic to the singular
cohomology of $X$.

\subsection{Filtrations}
\label{ss:filtrations}
We are mainly interested in degree bounds for the de Rham cohomology a closed variety $X\subseteq\A^n$, 
but in the course of the proof we also work with principal open subsets of the form $U=X\setminus\mZ(g)$.
The regular functions on $U$ have a power of~$g$ as denominator, so also their order becomes important.
To define the notions of degree and order in a precise and convenient way, we use the language of
(double-)filtrations.

Let $A=R/I$ be the coordinate ring of a closed variety $X\subseteq\A^n$. For
$f\in R$ we denote by~$\overline{f}$ its residue class in~$A$. We set
$$
\deg\overline{f}:=\min\{\deg h\,|\,h\in R,\ \overline{h}=\overline{f}\},
$$
where $\deg h$ denotes the total degree of $h$.
We have $\deg(\overline{f}\overline{g})\le \deg\overline{f}+\deg\overline{g}$, but note that 
this inequality may be strict.
We have the {\em filtration by degree}
$$
k= D^0A\subseteq\cdots\subseteq D^d A\subseteq D^{d+1}A\subseteq\cdots \subseteq A
$$
on $A$ given by
$$
D^d A := \left\{f\in A\,|\, \deg f\le d\right\},
$$
which satisfies
$$
D^d A\cdot D^e A\subseteq D^{d+e}A\quad\text{for all}\quad d,e\in\Z.
$$

We also consider {\em modules} $M$ over $A$ equipped with a filtration
$$
\cdots\subseteq F^d M\subseteq F^{d+1}M\subseteq\cdots \subseteq M,
$$
which we assume to be {\em compatible} with $D^\bullet A$, i.e.,
$D^{d} A\cdot F^{e} M\subseteq F^{d+e}M$ for all $d,e\in\Z$.
Given such a filtraton, one can define a degree by setting
$$
\deg x:=\min\{d\,|\, x\in F^dM\}\in\Z\quad\text{for}\quad x\in M.
$$
For a complex of $A$-modules $C^\bullet$, all equipped with a filtration $F^\bullet C^\bullet$, we
define the induced filtration on the cohomology $H^p(C^\bullet)$ by setting
\begin{equation}\label{eq:defFiltrCoh}
F^dH^p(C^\bullet):=\im(F^dC^p\cap\ker\ud\to H^p(C^\bullet)),
\end{equation}
where $\ud\colon C^p\to C^{p+1}$ denotes the differential of the complex.
For a filtered $k$-vector space $F^\bullet M$
we define
$$
\deg(M):=\inf\{d\in\Z \,|\, F^dM=M\}.
$$
Note that $\deg(M)=\infty$ if no such $d$ exists.

Now let $g\in R$ be a non-zerodivisor on $A$ and consider the localization~$A_{g}$.
The {\em order} of $h\in A_{g}$ with respect to $g$ is defined by
$$
\ord{g} h:=\min\{s\in\N\,|\,\exists f\in A\colon h=\frac{f}{g^s}\}.
$$
Sometimes we drop the index $g$, if it is clear from the context.
We have the {\em filtration by order}
$$
A = P^0A_{g}\subseteq\cdots\subseteq P^s A_{g}\subseteq P^{s+1}A_{g}\subseteq\cdots
\subseteq A_{g},
$$
where
$$
P^s A_{g} := \frac{1}{g^s} A=\left\{h\in A_{g}\,|\, \ord{g} h\le s\right\}.
$$
The definition of a convenient notion of degree in $A_{g}$ is problematic.
We illustrate this in an
\begin{example}\label{ex:degLoc}
Let $A=k[X_1,X_2,X_3]/(X_1^d X_2-X_3)$, $g=X_2$, and consider the element
\[
a:=\overline{X_3}/\overline{g}\in A_g.
\]
When we account for the denominator as negative degree, then $a$ should have degree zero. On the other hand
we have
\[
a=\overline{X_1}^d\in A_g,
\]
which is of degree $d$. We see that the degree depends on the representation of the element.
Furthermore, in this example, decreasing the order by 1 increases the degree of the numerator arbitrarily.
\end{example}
However, on a localization $R_g$ of the polynomial ring, we define the
\emph{degree} of $h\in R_g$ by
\begin{equation}\label{eq:degFrac}
\deg h:= \deg f-s\deg g,\quad\text{where}\quad h=\frac{f}{g^s},\quad f\in R,\quad s\in\N.
\end{equation}
Note that the degree can be arbitrarily small, so it induces an
unbounded \emph{filtration by degree} $D^\bullet R_g$. 
In $R_g$ we define the \emph{double filtration} $F^{\bullet,\bullet} R_g$ by
\[
F^{s,d} R_g := \{h\in R_g\,|\, \ord{g} h \le s,\ \deg h\le d\}
\]
Now, on $A_{g}$ we define
\[
F^{s,d} A_{g} := \pi(F^{s,d} R_g),
\]
where $\pi\colon R_g\to A_{g}$ is the natural projection.
Note that this double filtration depends on the representative $g\in R$ of $\overline{g}\in A$.
In our applications, we will always consider a fixed representative.
We have
\[
F^{s,d}A_{g}\subseteq F^{s+1,d}A_{g}\cap F^{s,d+1}A_{g} \qquad\text{and}\qquad
F^{s,d}A_{g}\cdot F^{s',d'}A_{g} \subseteq F^{s+s',d+d'}A_{g}.
\]

Now let $M$ be an $A_g$-module equipped with a double filtration $F^{\bullet,\bullet}M$, i.e., vector
subspaces $F^{s,d}M\subseteq M$ for $s,d\in\Z$ such that $F^{s,d}M\subseteq F^{s+1,d}M\cap F^{s,d+1}M$
for all $s,d\in\Z$. We assume that it is compatible with the filtration~$F^{\bullet,\bullet}A_g$, i.e.,
\[
F^{s,d}A_g \cdot F^{s',d'} M\subseteq F^{s+s',d+d'} M \quad\text{for all}\quad s,s',d,d'\in\Z. 
\]
Then we define the set
\begin{equation}
B(F^{\bullet,\bullet}M):=\{(s,d)\in\Z^2\,|\,F^{s,d}M=M\}.
\end{equation}
If the double filtration is understood from the context, we also simply write $B(M)$. Note that with
$(s,d)\in B(F^{\bullet,\bullet}M)$ we have $(s,d)+\N^2\subseteq B(F^{\bullet,\bullet}M)$.
A natural example is the localization $M=N_g$ of an $A$-module $N$. If $N$ has a compatible filtration $F^\bullet N$,
then $M$ carries a compatible double filtration $F^{\bullet,\bullet}M$ by setting
\begin{equation}\label{eq:defDoubleFiltrLocMod}
F^{s,d}M:=\frac{1}{g^s}F^{d+s\deg g} N.
\end{equation}

Similarly as in~\eqref{eq:defFiltrCoh}, a double filtration on a complex induces a double filtration 
on its cohomology. A double filtration on a double complex induces in a natural way a double filtation
on its total complex.

Particularly important for us are the modules of K\"ahler differential forms. 
The $A_g$-module $\Om_{A_g}^p$ is the localization $\Om_{A_g}^p=(\Om_{A}^p)_g$ of the $A$-module
$\Om_{A}^p$~\cite[Proposition 16.9]{eise:95}. 
We first define the filtration by degree on $\Om_A^p$ by
\[
D^d\Om_A^p :=\left\{\sum_{i_1<\cdots<i_p}f_{i_1\cdots i_p}\ud X_{i_1}\wedge\cdots\wedge\ud X_{i_p}\,|\, f_{i_1\cdots i_p}\in D^{d-p}A\right\}
\]
and note that it is compatible with the filtration $D^{\bullet}A$. We account for
the differentials in this definition of the degree, so that the differential map has degree zero, i.e.,
\[
\ud(D^d\Om_A^p)\subseteq D^d\Om_A^{p+1}.
\]
We have an induced degree function satisfying
\[
\deg f\ud X_{i_1}\wedge\cdots\wedge\ud X_{i_p} =\deg f + p \quad\text{for}\quad f\in A.
\]

Now, according to~\eqref{eq:defDoubleFiltrLocMod}, the filtration $D^\bullet\Om_A^p$ induces a compatible
double filtration on the localization $\Om_{A_g}^p$. We have
\begin{align*}
F^{s,d}\Om_{A_g}^p &= \frac{1}{g^s}D^{d+s\deg g} \Om_{A}^p \\
 &= \left\{\sum_{i_1<\cdots<i_p}f_{i_1\cdots i_p}\ud X_{i_1}\wedge\cdots\wedge\ud X_{i_p}\,|\, 
f_{i_1\cdots i_p}\in F^{s,d-p}A_g\right\},
\end{align*}
and call it the \emph{standard double filtration}.
We also have an order function on~$\Om_{A_g}^p$ satisfying
\[
\ord{g} f\ud X_{i_1}\wedge\cdots\wedge\ud X_{i_p}  =\ord{g} f\quad \text{for}\quad f\in A_g.
\]
With these definitions, the differential $\ud\colon \Om^p_{A_g}\to\Om^{p+1}_{A_g}$ satisfies
\begin{equation}\label{eq:difLocDegZero}
\ud(F^{s,d}\Om_{A_g}^p)\subseteq F^{s+1,d}\Om_{A_g}^{p+1}.
\end{equation}

Using these notations and conventions, we formulate an affine version of the GDD bound.
\begin{lemma}\label{lem:princOpen}
Let $f\in R$ and $U:=\A^n\setminus\mZ(f)$. Then, with respect to the standard double filtration,
we have
\[
p\cdot(1,1)\in B(\Hdr^p(U)).
\]
\end{lemma}
\begin{proof}
Denote by $\tilde{f}$ the generous homogenization $X_0^{d+1}f(X/X_0)$, where $d:=\deg f$.
Then we have $U=\proj^n\setminus\mZ(\tilde{f})$.
As stated in \S\ref{ss:knownCases}, each cohomology class in $\Hdr^p(U)$ is represented by
a differential form
$\tilde{\a}/\tilde{f}^p$,
where $\tilde{\a}$ is a homogeneous $p$-form on $\A^{n+1}$ of degree $\deg\tilde{\a}=p(d+1)$.
Dehomogenizing yields a form $\om=\a/f^p$ with $\deg\a\le p(d+1)$, hence $\omega\in F^{p,p}\Om_{R_f}^p$.
\end{proof}

\section{Zeroth Cohomology}
\label{se:zerothCoh}
In this section we discuss old and new results about the zeroth de Rham cohomology of a closed
variety $X\subseteq\A^n$. This is a somewhat special case, since $\Hdr^0(X)$ characterizes the
connected components of $X$ even if $X$ is singular.

We fix the notation for this section. Let $X=Z_1\cup\cdots \cup Z_t$ be the decomposition of $X$
into connected components, set $D:=\deg X$ and $D_i:=\deg Z_i$. By the results
of~\cite[\S3.1.2]{bus:09}, there is a direct product decomposition of the coordinate ring
\[
k[X]\simeq\prod_{i=1}^t k[Z_i],
\]
which corresponds to a maximal complete set of pairwise orthogonal idempotents $e_1,\ldots,e_t\in k[X]$.
Here, $e_i$ is the function which is equal to 1 on $Z_i$ and vanishes on $X\setminus Z_i$.
Furthermore, we have proved in~\cite{bus:09} that
$e_1,\ldots,e_t$ is a basis of $\Hdr^0(X)$ and that $\deg e_i\le d^{\Oh(n^2)}$, if $X$ is given by
polynomials of degree $\le d$. Due to Jelonek's version of the effective Nullstellensatz, we are
able to improve this bound.
\begin{proposition}\label{prop:boundConnComp1}
Let $X\subseteq\A^n$ be a closed variety of dimension $m$ and degree~$D$. Then $X$ has
$\dim \Hdr^0(X)$ connected components, and
$$
\deg(\Hdr^0(X))\le D^{m+1}.
$$
\end{proposition}
\begin{proof}
The result is known for $m=0$, so assume $m\ge 1$. We construct the idempotents $e_i$ as follows.
According to~\cite[Proposition 2.1]{bus:09}, each $Z_i$ can be defined by (many) polynomials
$f_{i\nu}$ of degree at most $D_i$. For $i\ne j$ we have $Z_i\cap Z_j=\emptyset$, hence the
$f_{j\nu}$ have no common zero in $Z_i$. By Theorem~\ref{thm:effNS} there exist polynomials $g_\nu$
with $\deg(g_\nu f_{j\nu})\le 2 D_i D_j^m$ such that
$$
\var_{ij}:=\sum_\nu g_\nu f_{j\nu}=1\quad\text{on}\quad Z_i.
$$
It is easy to see that the desired idempotents can be defined as
$$
e_i:=\prod_{j<i}\var_{ij}\cdot\prod_{j>i}\var_{ij}.
$$
Their degrees satisfy
$$
\deg e_i\le 2 D_i\sum_{j\ne i} D_j^m\le 2 D_i\big(\sum_{j\ne i} D_j\big)^m= 2D_i(D-D_i)^m.
$$
A small curve discussion shows that the last expression, as a function of $D_i$,  is maximal for
$D_i=\frac{D}{m+1}$, which implies
$\deg e_i\le\frac{2}{m+1}\big(\frac{m}{m+1}\big)^m D^{m+1}\le D^{m+1}$ for $m\ge 1$.
\end{proof}
This result gives a very good bound for small dimensions, say for curves. However, Koll\'ar's
effective Nullstellensatz for arbitrary ideals implies a bound which is quadratic in the degree
and therefore gives better results for larger dimensions.
\begin{proposition}\label{prop:boundConnComp2}
Let $X\subseteq\A^n$ be a closed variety of degree~$D$. Then
\begin{equation}\label{eq:quadBound}
\deg(\Hdr^0(X))\le \frac{n+1}{4}D^2.
\end{equation}
\end{proposition}
\begin{proof}
By Theorem~\ref{thm:effNSIdeals} there exist polynomials $\var_{ij}\in I(Z_i)$ and $\psi_{ij}\in I(Z_j)$
for $i\ne j$ such that
$$
\deg(\var_{ij}),\ \deg(\psi_{ij})\le (n+1) D_i D_j\quad\text{and}\quad
\var_{ij}+\psi_{ij}=1.
$$
Now the desired idempotents can be defined as
$$
e_i:=\prod_{j<i}\var_{ji}\cdot\prod_{j>i}\psi_{ij}.
$$
Their degrees satisfy
\begin{equation*}
\deg e_i\le (n+1) D_i\sum_{j\ne i} D_j= (n+1)D_i(D-D_i)\le (n+1)\left(D/2\right)^2.\qedhere
\end{equation*}
\end{proof}
\begin{remark}
In~\cite{sch:10} we have proved that for a hypersurface $X$ the factor $n+1$ in~\eqref{eq:quadBound}
can be dropped.
\end{remark}

\begin{example}
This example shows that the bound~\eqref{eq:quadBound} is sharp up to the factor $n+1$.
It is derived from Example 2.3 of~\cite{koll:88}, which goes back to Masser, Philippon, and
Brownawell~\cite{brow:87}.

Let $d\ge 1$. Consider the polynomials
$$
f_1:=X_1,\quad  f_2:=X_2X_3^{d-1}-1,\quad f_3:=X_1X_3^{d-1}-X_2^d,
$$
and set $Z_1:=\mZ(f_1,f_2)$ and $Z_2:=\mZ(f_3)$ in $\A^3$. Clearly, both~$Z_1$
and $Z_2$ are smooth irreducible varieties of degree $d$
that do not intersect. Consequently, they are the connected components of $X:=Z_1\cup Z_2$, and
$D:=\deg X=2d$. Now consider the projective closure $\overline{X}=\overline{Z}_1\cup\overline{Z}_2$.
Let $F_i\in k[X_0,\ldots,X_3]$ denote the homogenization of $f_i$, i.e,
$$
F_1=X_1,\quad  F_2:=X_2X_3^{d-1}-X_0^d,\quad F_3=X_1X_3^{d-1}-X_2^d.
$$
One easily checks that $I(\overline{Z}_1)=(F_1,F_2)$ and $I(\overline{Z}_2)=(F_3)$.
Now let $e_1,e_2\in k[X_1,X_2,X_3]$ denote the idempotents of $Z_1$ and $Z_2$, $\d:=\max\{\deg e_1,\deg e_2\}$,
and $E_1,E_2$ their homogenizations w.r.t.~degree $\d$, i.e., $E_i=X_0^\d e_i(X/X_0)$. Then
we have
$$
E_1+E_2=X_0^\d\ \text{on}\ \overline{X},\quad 
E_1=0\ \text{on}\ \overline{Z}_2,\quad
E_2=0\ \text{on}\ \overline{Z}_1,
$$
hence $X_0^\d\in I(\overline{Z}_1)+I(\overline{Z}_2)$. It follows that its image in
$$
k[X_0,\ldots,X_3]/(F_1,F_2,F_3,X_3-1)\simeq k[X_0]/(X_0^{d^2})
$$
is also zero, so that
$$
\max\{\deg e_1,\deg e_2\}=\d\ge d^2=D^2/4.
$$
\end{example}

\section{Some Reductions}
Our aim is to show that it is sufficient to prove the claimed bounds for certain irreducible
hypersurfaces. In this section we first reduce to the irreducible case, and then effectively
construct a cover of $X$ by affine open patches which are isomorphic to hypersurfaces.
In the next section, we will show how degree bounds on those patches yield a bound for
$\Hdr^\bullet(X)$.

The reduction to the irreducible case follows from the effective characterization of the
connected components discussed in \S\ref{se:zerothCoh}.

\begin{corollary}\label{cor:redIrred}
If $X\subseteq\A^n$ is a smooth closed variety of degree $D$ with irreducible components $Z_i$
and $N:=\min\left\{(n+1)D^2/4,D^{m+1}\right\}$, then
$$
\deg(\Hdr^\bullet(X))\le\max_i\{\deg(\Hdr^\bullet(Z_i))\}+N.
$$
\end{corollary}
\begin{proof}
Since the irreducible components of $X$ coincide with its connected components, the restrictions
of differential forms to the $Z_i$ induce an isomorphism
\begin{equation}\label{eq:cohDirectSum}
\Hdr^\bullet(X)\simeq\bigoplus_i \Hdr^\bullet(Z_i)).
\end{equation}
Let $e_i\in k[X]$ denote the idempotent corresponding to $Z_i$. Since $e_i$ is 1 on $Z_i$ and 0
on $X\setminus Z_i$, the embedding $\Hdr^\bullet(Z_i)\hookrightarrow\Hdr^\bullet(X)$ is induced by
the map $\om\mapsto e_i\om$, so that Propositions~\ref{prop:boundConnComp1}
and~\ref{prop:boundConnComp2} imply the claim.
\end{proof}

It is well known that each irreducible variety $X$ is birational to a hypersurface~$Y$. By
Zariski's Main Theorem, there exist open dense subsets $U\subseteq X$ and $V\subseteq Y$ which
are isomorphic. Now we make this construction effective and obtain the result that degree and
order bounds for $\Hdr^\bullet(V)$ imply such bounds for $\Hdr^\bullet(U)$.

The following lemma essentially consists of the construction of a {\em geometric
resolution}~\cite{ghmp:95,par:95}.
\begin{lemma}\label{lem:biratHyper}
Let $X\subseteq \A^n$ be a closed irreducible subvariety of dimension~$m<n$
and $\Xi\subseteq X$ be a finite subset. Then there exists a linear coordinate
transformation after which $X$ is in Noether normal position with algebraically
independent variables $X_1,\ldots,X_m$ and a polynomial
$f\in k[X_1,\ldots,X_{m+1}]$ such that
\begin{enumerate}[(i)]
\item $f$ is irreducible and monic in $X_{m+1}$,
\item $\deg f\le\deg X$,
\item $Y:=\mZ(f)\subseteq\A^{m+1}$ is the closure of the image of $X$ under the projection
$\pi\colon\A^n\to\A^{m+1}$, $(x_1,\ldots,x_n)\mapsto(x_1,\ldots,x_{m+1})$,
\item with $g:=\partder{f}{X_{m+1}}$ the projection $\pi$ restricts to an
isomorphism
$$
U:=X\setminus \mZ(g)\stackrel{\simeq}{\longrightarrow} V:=Y\setminus\mZ(g),\quad
\text{and}
$$
\item $g(\xi)\ne 0$ for all $\xi\in\Xi$.
\end{enumerate}
\end{lemma}
\begin{proof}
Let $A=R/I(X)$ be the coordinate ring of $X$. A linear coordinate transformation
brings $X$ into Noether normal position, i.e., we assume that the ring extension
$k[X_1,\ldots,X_m]\hookrightarrow A$ is integral. Let $r:=n-m$. For
$\la=(\la_1,\ldots,\la_r)\in k^r$ let $u_{\la}:=\la_1 X_{m+1}+\cdots+\la_rX_{n}$.
Then there exists a monic polynomial $f_{\la}\in k[X_1,\ldots,X_m][T]$ with
$f_{\la}(u_{\la})=0$ in $A$. Since $A$ is a domain, we can assume
$f_{\la}$ to be irreducible. Let $\pi_{\la}\colon X\to \A^{m+1}$ be the map
$(x_1,\ldots,x_n)\mapsto(x_1,\ldots,x_m,u_{\la})$. Obviously,~$f_{\la}$
vanishes on the irreducible hypersurface $\pi_{\la}(X)$ of $\A^{m+1}$, so it
is a reduced equation for it. It follows
$\deg f_{\la}=\deg \pi_{\la}(X)\le\deg X$~\cite[Lemma 2]{hei:83}.
Now we consider $\la_1,\ldots,\la_r$ as variables and argue as above over the
field $k_{\la}:=\overline{k(\la_1,\ldots,\la_r)}$. It is shown in~\cite{gls:01,sas:95}
that $f_{\la}\in k[\la,X_1,\ldots,X_m,T]$. Now set $g_{\la}:=\partder{f_{\la}}{T}$.
Differentiating the equation $f_{\la}(u_{\la})=0$ with respect to $\la_i$, the
chain rule yields
$$
\partder{f_{\la}}{\la_i}(u_{\la})+\partder{f_{\la}}{T}(u_{\la})\partder{u_{\la}}{\la_i}
=\partder{f_{\la}}{\la_i}(u_{\la})+g_{\la}(u_{\la})X_{m+i}=0,
\quad 1\le i\le r.
$$
We choose $\la\in k^r$ such that $g_{\la}(\xi_1,\ldots,\xi_m,u_{\la}(\xi))\ne 0$
for all $\xi=(\xi_1,\ldots,\xi_n)\in\Xi$ and denote $f:=f_{\la}$, $g:=g_{\la}$,
$w_i:=\partder{f_{\la}}{\la_i}$.
By another linear coordinate transformation we can assume $u_\la=X_{m+1}$. It
follows that the map
$$
\mZ(f)\setminus\mZ(g)\to X\setminus\mZ(g),\quad
x=(x_1,\ldots,x_{m+1})\mapsto \left(x,-\frac{w_2(x)}{g(x)},\ldots,-\frac{w_r(x)}{g(x)}\right)
$$
is an inverse of the projection, which concludes the proof.
\end{proof}

\begin{remark}
If $X\subseteq\A^n$ is smooth, then $V$ from Lemma~\ref{lem:biratHyper} is also
smooth.
\end{remark}

\begin{corollary}\label{cor:transHyperPatch}
In the situation of Lemma~\ref{lem:biratHyper}, consider the open subsets
$U':=U\setminus\mZ(h)\subseteq X$, where $h\in R$, and
$V':=\pi(U')\subseteq Y$. Then there exists $H\in k[X_1,\ldots,X_{m+1}]$ with
$\deg H\le\deg h\deg X$ such that $k[V']=k[Y]_{gH}$. Moreover, the
isomorphism $\varphi\colon k[Y]_{gH}=k[V']\stackrel{\simeq}{\to} k[U']=k[X]_{gh}$ 
satisfies
\[
\varphi(F^{s,d} k[Y]_{gH}) \subseteq F^{s(\deg h+1),d} k[X]_{gh}.
\]
\end{corollary}
\begin{proof}
By the proof of Lemma~\ref{lem:biratHyper}, there exist $w_i\in k[X_1,\ldots,X_{m+1}]$
of degree at most $\deg f$ such that
\[
X_{m+i}=-w_i/g\quad\text{in}\quad k[X]_g\quad\text{for}\quad 1\le i\le r=n-m.
\]
Thus, the preimage of $h$ under the isomorphism $k[Y]_g\stackrel{\simeq}{\longrightarrow} k[X]_g$
is the residue class of the rational function
$$
\tilde{h}:=h(X_1,\ldots,X_m,-\frac{w_1}{g},\ldots,-\frac{w_r}{g})=H(X_1,\ldots,X_{m+1})/g^{\deg h},
$$
where $H\in k[X_1,\ldots,X_{m+1}]$. It follows $k[V']=k[Y]_{gH}$.
Since $\deg(\frac{w_i}{g})\le 1$, we have
$$
\deg H=\deg\tilde{h}+\deg h\deg g\le\deg h(1+\deg g)=\deg h\deg f\le\deg h\deg X.
$$
An element of $F^{s,d} k[Y]_{gH}$ is the residue class of a rational function $\frac{a}{(gH)^s}$, where
$a\in k[X_1,\ldots,X_{m+1}]$ with $\deg a\le d+s(\deg g+\deg H)$. The isomorphism $\varphi$ 
identifies $H$ with $g^{\deg h} h$, so in $k[X]_{gh}$ we have
\[
\frac{a}{(gH)^s}\equiv \frac{a}{(g\cdot g^{\deg h} h)^s}=\frac{ah^{s\deg h}}{(gh)^{s(\deg h+1)}}.
\]
Since
\begin{align*}
\deg\left(\frac{a}{(g\cdot g^{\deg h} h)^s}\right) &= \deg a-s(\deg h+(\deg h+1)\deg g) \\
 &\le d+s(\deg g+\deg h\deg f-\deg h-\deg h \deg g-\deg g) \\
 &\le d+s\deg h\underbrace{(\deg f-1-\deg g)}_{=0} = d,
\end{align*}
we have $\varphi(\overline{\frac{a}{(gH)^s}})\in F^{s(\deg h+1),d} k[X]_{gh}$.
\end{proof}

\begin{corollary}\label{cor:biratHyperCover}
Let $X\subseteq \A^n$ be a closed irreducible variety of dimension~$m$.
Then there exist principal open subsets $U_0,\ldots,U_m$ covering $X$,
such that each~$U_i$ is isomorphic to an open
subset of a hypersurface described as in Lemma~\ref{lem:biratHyper}.
\end{corollary} 
\begin{proof}
We apply Lemma~\ref{lem:biratHyper} successively to construct the $U_i$. We
start with an arbitrary one-point set $\Xi_0=\{\xi\}\subseteq X$ to
obtain $U_0$. 
Having constructed $U_0,\ldots, U_{i-1}$, let $\Xi_i$ contain a point from each
irreducible component of $X\setminus \bigcup_{0\le j<i} U_j$, so that one obtains
$U_i$ with
$$
\dim\Big(X\setminus \bigcup_{j=0}^{i-1} U_j\Big)>\dim\Big(X\setminus\bigcup_{j=0}^{i} U_j\Big).
$$
It follows that $\dim X\setminus \bigcup_{0\le j\le m} U_{j}=-1$, hence
$\bigcup_{i=0}^m U_i=X$.
\end{proof}

\section{Effective \v{C}ech and Hypercohomology}
In this section we show how the cohomologies of the open patches of Corollary~\ref{cor:biratHyperCover}
fit together to give effective descriptions of the cohomology of $X$. This is done via effective
hypercohomology, for which we first need to make \v{C}ech cohomology effective.

So let $X\subseteq\A^n$ be a closed variety, and let $\mF$ be a coherent sheaf on $X$.  Since~$X$
is affine, we have
\begin{equation}\label{eq:vanishing}
H^0(X,\mF)=\G(X,\mF),\quad H^q(X,\mF)=0\quad\text{for}\quad q>0.
\end{equation}
Now let $A:=k[X]$ be the coordinate ring of $X$, and let $M$ be a finitely generated $A$-module
such that $\mF=\tilde{M}$. Furthermore, let $\mU=\{U_i\,|\,0\le i\le t\}$ be a cover of $X$ by
principal open subsets $U_i=X\setminus\mZ(g_i)$, where $g_i\in R$ is a non-zerodivisor on $A$.
By \S\ref{ss:sheafCohomology},
the sheaf cohomology of $\mF$ can be computed as the \v{C}ech cohomology with respect to~$\mU$.
Using~\S\ref{ss:cohSheaves}, the \v{C}ech complex $C^\bullet=C^\bullet(\mU,\mF)$ can be explicitly
descibed as follows. For a set of indices $0\le i_0,\ldots,i_q\le t$ set
$g_{i_0\cdots i_q}:=g_{i_0}\cdots g_{i_q}$. Then
$$
C^q=\bigoplus_{0\le i_0<\ldots<i_q\le t}M_{g_{i_0\cdots i_q}},\quad 0\le q\le t.
$$
The properties~\eqref{eq:vanishing} mean that the augmented complex
$$
0\longrightarrow C^{-1}:=M\stackrel{\d^{-1}}{\longrightarrow}C^0
\stackrel{\d^0}{\longrightarrow}C^1\stackrel{\d^1}{\longrightarrow}
\cdots\stackrel{\d^{t-1}}{\longrightarrow}C^t\longrightarrow 0
$$
is exact, where $\d^{-1}$ is induced by the restriction. Note that $\d^{-1}$ yields the isomorphism
$M\simeq H^0(X,\mF)$.

The point of this section is to study degree and order bounds for a preimage under $\d^\bullet$. So
we consider a filtration
$$
\cdots\subseteq F^d M\subseteq F^{d+1}M\subseteq\cdots \subseteq M
$$
on $M$ which is compatible with $D^\bullet A$.
Recall from \S\ref{ss:basics} that there is an induced double filtration
$F^{\bullet,\bullet} M_{g_{i_0\cdots i_q}}$. These double filtrations extend to the spaces $C^q$ by setting
$$
F^{s,d} C^q:=\bigoplus_{0\le i_0<\cdots<i_q\le t} F^{s,d} M_{g_{i_0}\cdots g_{i_q}}
$$
Since the restriction does increase neither degree nor order,
we have $\d^q(F^{s,d} C^q)\subseteq F^{s,d} C^{q+1}$.

\begin{lemma}\label{lem:chechCoh}
Let $\deg X=D$, $\deg g_i\le d_1$, and $N:=2D(sd_1)^m$.
Then
\[
F^{s,d}C^q\cap\ker \d^q \subseteq \d^{q-1}(F^{s,d+N}C^{q-1}) \quad\text{for all}\quad 0\le q\le t.
\]
\end{lemma}
\begin{proof}
For $\om=(\om_{i_0\cdots i_q})_{i_0<\cdots <i_q}\in F^{s,d} C^q\cap\ker \d^q$ we write
$\om_{i_0\cdots i_q}=\frac{\a_{i_0\cdots i_q}}{g_{i_0\cdots i_q}^s}$, where
$\a_{i_0\cdots i_q}\in M$ with $\deg \a_{i_0\cdots i_q}\le d+s\deg g_{i_0\cdots i_q}$.
The assumption $\d^q(\om)=0$ implies
\begin{equation}\label{eq:closed}
\sum_\nu (-1)^\nu g_{i_\nu}^s\a_{i_0\cdots\hat{i_\nu}\cdots i_{q+1}}=0\quad\text{for all}
\quad 0\le i_0<\cdots <i_{q+1}\le t.
\end{equation}
Since $\mU$ is a cover of $X$, the $g_0^s,\ldots,g_t^s$ have no
common zero in $X$. Thus, by Theorem~\ref{thm:effNS} there are
$h_0,\ldots, h_t\in A$ with $\sum_i h_i g_i^s=1$ and $\deg(h_i g_i^s)\le N$.
Now we let $\eta=(\eta_{i_0\cdots i_{q-1}})\in C^{q-1}$ with
$$
\eta_{i_0\cdots i_{q-1}}:=\frac{1}{g_{i_0\cdots i_{q-1}}^s}\sum_{i=0}^t h_i\a_{i,i_0\cdots i_{q-1}}
\quad\text{for all}\quad 0\le i_0<\cdots <i_{q-1}\le t.
$$
Here we define $\a_{i,i_0\cdots i_{q-1}}$ to be zero, if $i\in\{i_0\cdots i_{q-1}\}$,
and $\eps\a_{j_0\cdots j_q}$ otherwise, where $\{i_0,\ldots, i_{q-1}, i\}=\{j_0<\cdots <j_q\}$
and $\eps$ is the sign of the permutation sorting $(i_0,\ldots, i_{q-1}, i)$.
Obviously, we have $\ord{}\eta_{i_0\cdots i_{q-1}}\le s$. To bound its degree,
note that for all $i$
$$
\deg(h_i\a_{i,i_0\cdots i_{q-1}})\le N-s\deg g_i+d+s\deg g_{i,i_0\cdots i_{q-1}}
=N+d+s\deg g_{i_0\cdots i_{q-1}},
$$
which implies $\deg\eta\le N+d$.

Finally, we have
\begin{align*}
\d^{q-1}(\eta)_{i_0\cdots i_q} & = \sum_{\nu=0}^q(-1)^\nu \eta_{i_0\cdots\hat{i_\nu}\cdots i_q}|U_{i_0\cdots i_q} \\
 & = \sum_{\nu=0}^q(-1)^{\nu}\frac{1}{g_{i_0\cdots\hat{i_\nu}\cdots i_q}^s}
\sum_{i=0}^t h_i\a_{i,i_0\cdots\hat{i_\nu}\cdots i_q} \\
 & = \frac{1}{g_{i_0\cdots i_q}^s}\sum_{i=0}^th_i\sum_{\nu=0}^q(-1)^{\nu}
g_{i_\nu}^s \a_{i,i_0\cdots\hat{i_\nu}\cdots i_q} \\
 & \stackrel{\eqref{eq:closed}}{=} \frac{1}{g_{i_0\cdots i_q}^s}\sum_{i=0}^th_i
g_i^s\a_{i_0\cdots i_q} \\
 & = \om_{i_0\cdots i_q},
\end{align*}
hence $\d^{q-1}(\eta)=\om$.
\end{proof}

Next we will bound the degree and order of the hypercohomology of a complex of coherent sheaves,
given that we have degree and order bounds for the cohomologies of the open patches. This can be
formulated in a general setting.
\begin{lemma}\label{lem:hyperCoh}
Let $C^{\bullet,\bullet}$ be a first quadrant double complex, where each $C^{p,q}$ is equipped with a
double filtration $F^{\bullet,\bullet}C^{p,q}$. Then we have
$$
\bigcap_{p+q=\ell} B(F^{\bullet,\bullet}H^p(C^{\bullet,q})) \subseteq B(H^\ell(\tot^\bullet(C^{\bullet,\bullet})))\quad
\text{for all}\quad\ell\ge 0.
$$
\end{lemma}
\begin{proof}
The second spectral sequence ${}_{II}E_r$ of the double complex $C^{\bullet,\bullet}$ has first term
$$
{}_{II}E_1^{p,q}=H^p(C^{\bullet,q}).
$$
Note that the double filtration on $C^{\bullet,\bullet}$ induces a double filtration on the spectral
sequence. Each term of the spectral sequence is a cohomology of the previous one realized as a
subquotient, thus
$$
B(F^{\bullet,\bullet} {}_{II}E_r^{p,q})\subseteq B(F^{\bullet,\bullet} {}_{II}E_{r+1}^{p,q})\quad\text{for all}\quad r\ge 1.
$$
Moreover, since the double complex is bounded, the spectral sequence
collapses at a finite level, i.e., there exists $r$ such that
${}_{II}E_r=\, {}_{II}E_\infty$. By convergence we have
$$
H^\ell(\tot^\bullet(C^{\bullet,\bullet}))\simeq \bigoplus_{p+q=\ell}\, {}_{II}E_\infty^{p,q}.
$$
Finally, since this isomorphism respects the double filtrations $F^{\bullet,\bullet}C^{p,q}$, we
conclude
$$
B(H^\ell(\tot^\bullet(C^{\bullet,\bullet})))=\bigcap_{p+q=\ell} B(F^{\bullet,\bullet} {{}_{II}E_\infty^{p,q}}),
$$
which implies the claim.
\end{proof}
Now consider a bounded complex of coherent sheaves
$$
\mF^\bullet\colon 0\longrightarrow\mF^0\stackrel{\ud}{\longrightarrow}
\mF^1\stackrel{\ud}{\longrightarrow}\cdots
\stackrel{\ud}{\longrightarrow} \mF^u\longrightarrow 0
$$
on $X$, and let $M^p$ be finitely generated $A$-modules with $\mF^p=\tilde{M}^p$.
Furthermore, let $\mU=\{U_i\,|\,0\le i\le t\}$ be a cover of $X$ by principal open
subsets $U_i=X\setminus\mZ(g_i)$ as above. Recall from \S\ref{ss:hyperCohomology} that the
hypercohomology $\H^\ell(X,\mF^\bullet)$ can be computed in two ways, namely as the cohomology of
the total complex $\tot^\bullet(C^{\bullet,\bullet}(\mU,\mF))$, and as the cohomology of
the complex of global sections $M^\bullet$. Of course, the latter is much simpler, which is usually
taken as an argument that one does not need hypercohomology in the affine setting. However, we
want to go through these arguments to bound degrees. Assume that we have filtrations
$$
\cdots\subseteq F^d M^p\subseteq F^{d+1}M^p\subseteq\cdots \subseteq M^p,
$$
such that the differential $\ud$ satisfies
\[
\ud (F^{s,d}M_g^p) \subseteq F^{s+1,d}M_g^p
\]
with respect to the induced double filtration on the local sections.
As above, the \v{C}ech double complex and its total complex inherit this double filtration.

\begin{theorem}\label{thm:effHyperCoh}
Let $\deg X=D$ and $\deg g_i\le d_1$. Then for all $\ell$ and all
$$
(s,d)\in\bigcap_{p+q=\ell}\bigcap_{i_0<\cdots <i_q}B(H^p(M^\bullet_{g_{i_0\cdots i_q}}))
$$
we have
$$
\deg(H^\ell(M^\bullet))\le d+2D(\ell+1)(s+\ell)^m d_1^m.
$$
\end{theorem}
\begin{proof}
Fix $\ell$ and set $N:=2D(s+\ell)^m d_1^m$.
By Lemma~\ref{lem:hyperCoh} each cohomology class
$c\in H^\ell(\tot^\bullet(C^{\bullet,\bullet}))$ can be represented by an element
$\om=(\om_{0,\ell},\om_{1,\ell-1},\ldots,\om_{\ell,0})$, where
$$
\om_{p,q}\in F^{s,d} C^{p,q}\quad\text{for all}\quad p+q=\ell.
$$
Since $d^\tot(\om)=0$, we have in particular $\d^\ell(\om_{0,\ell})=0$, so $\om_{0,\ell}$
defines an element in $H^\ell(X,\mF^0)$. By Lemma~\ref{lem:chechCoh} there exists
$$
\eta\in F^{s,d+N} C^{0,\ell-1}
$$
with $\d^{\ell-1}(\eta)=\om_{0,\ell}$. Consider
$$
\om':=(0,\om_{1,\ell-1}+(-1)^\ell\ud\eta,\om_{2,\ell-2},\ldots,\om_{\ell,0}).
$$
Since $\om-\om'=(\om_{0,\ell},(-1)^{\ell-1}\ud\eta,0,\ldots,0)=\ud^\tot(\eta,0,\ldots,0)$,
$\om$ and $\om'$ define the same cohomology class $c$.
Note that $\om_{1,\ell-1}+(-1)^\ell\ud\eta\in F^{s+1,d+N} C^{1,\ell-1}$.
Continuing this way, after $\ell$ steps we have a representative of~$c$ of the
form $\om''=(0,\ldots,0,\om_{\ell,0}'')$, where
$$
\om_{\ell,0}''\in F^{s+\ell,d+\ell N} C^{\ell,0}.
$$
The closedness of $\om''$ means that $\d^0(\om_{\ell,0}'')=0$ and $\ud(\om_{\ell,0}'')=0$.
Applying Lemma~\ref{lem:chechCoh} once more yields $\a\in F^{d+(\ell+1)N} M^\ell$
with $\d^{-1}(\a)=\om_{\ell,0}''$. Moreover, $\ud(\a)=0$, hence $\a$ defines a
class in $H^\ell(M^\bullet)$, which is the image of $c$ under the isomorphism
$H^\ell(\tot^\bullet(C^{\bullet,\bullet}))\simeq H^\ell(M^\bullet)$.
\end{proof}

We close this section by discussing the case of an open subset of $X:=\A^n$.
Let $Z=\mZ(f_0,\ldots,f_r)\subseteq X$ be a closed subvariety and consider its
complement $W:=X\setminus Z$. If $r>0$, then $W$ is not affine, and as a result its de Rham
cohomology $\Hdr^\bullet(W)$ cannot be computed by global sections, so we have to use
hypercohomology instead. The open subsets $U_i:=X\setminus\mZ(f_i)$, $0\le i\le r$, 
form an open cover of $W$ with corresponding \v{C}ech-de Rham double complex
$$
C^{p,q}=\bigoplus_{i_0<\cdots <i_q}\Om_W^p(U_{i_0\cdots i_q})\quad\text{for all}\quad p,q\ge 0,
$$
where $U_{i_0\cdots i_q}=U_{i_0}\cap\cdots\cap U_{i_q}$.
The differentials $\ud\colon C^{p,q}\to C^{p+1,q}$ are induced by the exterior derivatives, and
the differentials $\d^q\colon C^{p,q}\to C^{p,q+1}$ are the \v{C}ech differentials.
The de Rham cohomology of $W$ is the cohomology of the total complex $\tot^\bullet(C^{\bullet,\bullet})$, i.e.,
$$
\Hdr^\ell(W)=\H^\ell(W,\Om_W^\bullet)=H^\ell(\tot^\bullet(C^{\bullet,\bullet})).
$$
\begin{lemma}\label{lem:boundHyperComplement}
Equip each $C^{p,q}$ with the shifted double filtration $S^{s,d}C^{p,q}:=F^{s-q,d-q}C^{p,q}$
for $s,d\in\N$. Then, with respect to the induced double filtration
on the total complex, we have 
$$
\ell\cdot (1,1)\in B(S^{\bullet,\bullet}\Hdr^\ell(W)).
$$
\end{lemma}
\begin{proof}
By Lemma~\ref{lem:princOpen} we have with respect to the shifted double filtration
$$
\ell\cdot (1,1)\in B(S^{\bullet,\bullet}\Hdr^p(U_{i_0\cdots i_q}))\quad\text{for all}\quad p+q=\ell,\ 0\le i_0<\cdots <i_q\le r.
$$
Applying Lemma~\ref{lem:hyperCoh} implies the claim.
\end{proof}

In the case of a complete intersection we are able to give a more succinct description of
$\Hdr^\bullet(W)$ using sheaf cohomology as follows.
\begin{lemma}\label{lem:deRhamLocal}
Let $\dim Z=n-r-1$. Then, for $\ell>0$, the projection map 
$$
\tot^\ell(C^{\bullet,\bullet})=C^{\ell-r,r}\oplus\cdots\oplus C^{\ell,0}\twoheadrightarrow
C^{\ell-r,r}=\Om_{A_{f_0\cdots f_r}}^{\ell-r}
$$
induces an isomorphism
\begin{equation}
\Hdr^\ell(W)\stackrel{\simeq}{\longrightarrow} H^{\ell-r}(H^r(W,\Om_W^\bullet)).
\end{equation}
\end{lemma}
\begin{proof}
Let $\mF$ be a locally free sheaf on $X$. The long exact sequence of local cohomology is
$$
\cdots\to H^{q-1}(X,\mF)\to H^{q-1}(W,\mF|W)\to H_Z^{q}(X,\mF)\to H^{q}(X,\mF)\to\cdots.
$$
Since $H^{q-1}(X,\mF)=0$ for $q>1$, we conclude
\begin{equation}\label{eq:isoSheafLocalCoh}
H^{q-1}(W,\mF|W)\simeq H_Z^{q}(X,\mF) \quad\text{for}\quad q>1.
\end{equation}
Since $Z$ is a complete intersection of codimenion $r+1$, it follows from~\cite[Theorem 3.8]{har:67},
that $H_Z^q(X,\mF)=0$ for $q< r+1$. Furthermore, since $W=X\setminus Z$ can be covered by $r+1$
affine open subsets, \v{C}ech cohomology implies $H_Z^q(X,\mF)=H^{q-1}(W,\mF|W)=0$ for $q>r+1$.
Thus
\begin{equation}\label{eq:localCohLacunary}
H_Z^{q}(X,\mF)=0 \quad\text{for}\quad q\ne r+1.
\end{equation}
The beginning of the long exact sequence together with~\eqref{eq:localCohLacunary} implies
\begin{equation}\label{eq:globalSections}
H^0(W,\mF|W)= H^0(X,\mF).
\end{equation}
Now, the first spectral sequence of hypercohomology has first term
$$
{}_{I}E_1^{pq}=H^q(W,\Om_W^p)\quad\Rightarrow\quad \Hdr^{p+q}(W).
$$
By~\eqref{eq:isoSheafLocalCoh}, \eqref{eq:localCohLacunary}, and~\eqref{eq:globalSections} we have
$$
{}_{I}E_1^{pq}=\left\{
\begin{array}{rl}
H^0(X,\Om_X^p) & \text{if } q=0, \\
H^{r}(W,\Om_W^p) & \text{if } q=r, \\
0 & \text{otherwise}.
\end{array}\right.
$$
Since $\Hdr^p(X)=H^p(H^0(X,\Om_X^p))=0$ for $p>0$, the second page ${}_{I}E_2$ has differential zero,
hence ${}_{I}E_2= {}_{I}E_\infty$, and we conclude
\begin{equation*}
\Hdr^\ell(W)={}_{I}E_2^{\ell-r,r}=H^{\ell-r}(H^{r}(W,\Om_W^\bullet))\quad\text{for}\quad \ell>0.\qedhere
\end{equation*}
\end{proof}

Lemmas~\ref{lem:boundHyperComplement} and~\ref{lem:deRhamLocal} imply
\begin{corollary}\label{cor:boundLocalCoh}
$$
(\ell-r)\cdot (1,1)\in B( H^{\ell-r}(H^r(W,\Om_W^\bullet))).
$$
\end{corollary}

\section{Effective Gysin Sequence}\label{se:effGysin}
The main tool in our proof is the Gysin sequence which is the following
\begin{theorem}\label{thm:gysinSequ}
Let $X$ be a smooth irreducible variety and $Z\subseteq X$ a smooth closed equidimensional
subvariety of codimension $r$. Then there is an exact sequence
$$
\cdots\rightarrow \Hdr^p(X)\to \Hdr^p(X\setminus Z)\stackrel{\Res}{\to}
\Hdr^{p-2r+1}(Z)\to \Hdr^{p+1}(X)\to\cdots
$$
\end{theorem}

Let us first record an easy consequence of the Gysin sequence for the case that $X$ is
the affine space $\A^n$. Since $\Hdr^p(X)=0$ for $p>0$,
Theorem~\ref{thm:gysinSequ} implies
\begin{corollary}\label{cor:ResIso}
For a smooth closed equidimensional variety $Z\subseteq \A^n$ of codimension $r$ the residue map
$$
\Res\colon\Hdr^p(\A^n\setminus Z)\stackrel{\simeq}{\rightarrow} \Hdr^{p-2r+1}(Z)
$$
is an isomorphism for all $p>0$.
\end{corollary}

An algebraic proof of Theorem~\ref{thm:gysinSequ} is given in~\cite{har:70}.
Using this idea we will prove an effective version of it, i.e., we will describe
the map Res
explicitly on the level of differential forms, so that we can control its effect
on their degrees.

However, our version only deals with the case of a smooth codimension 2 complete intersection of
a very special type. In particular, we will study the following case. Denote $A:=k[X_0,\ldots,X_n]$,
let $f,g\in R:=k[X_1,\ldots,X_n]$, and consider the regular sequence $f_0:=gX_0-1$, $f_1:=f$. Set
$X:=\A^{n+1}$, $Z:=\mZ(f_0,f_1)$ and denote its coordinate ring by $B:=k[Z]$. Assume that~$f$ is
irreducible and that $\partder{f}{X_n}|g$. Note that these assumptions imply that $Z$ is a smooth
complete intersection in $X$ of codimension 2, and its vanishing ideal is $I:=I(Z)=(f_0,f_1)$,
hence $B=A/(f_0,f_1)$.

Since in our case the complement $W:=X\setminus Z$ is not affine, its de Rham cohomology
$\Hdr^\bullet(W)$ is described as in the last section. Recall that by Lemma~\ref{lem:deRhamLocal}
we have an isomorphism
\begin{equation}\label{eq:isodeRhamLocal}
\Hdr^{p+1}(W)\stackrel{\simeq}{\longrightarrow} H^{p}(H^1(W,\Om_W^\bullet))=
H^{p}\left(\frac{\Om_{A_{f_0f_1}}^\bullet}{\Om_{A_{f_0}}^\bullet+\Om_{A_{f_1}}^\bullet}\right)\quad\text{for}\quad p\ge 0.
\end{equation}
The main result of this section is
\begin{theorem}\label{thm:effGysin}
Let $f_0,f_1\in A$ be as above, and denote $d_0:=\deg f_0, d_1:=\deg f_1$.
Then, under the identification~\eqref{eq:isodeRhamLocal}, the map
$$
\Res\colon\Hdr^{p+1}(W)\to\Hdr^{p-2}(Z)\quad\text{for}\quad p\ge 0
$$
is induced by a map $\Om_{A_{f_0f_1}}^{p}\to\Om_B^{p-2}$ which takes a $p$-form $\om=\frac{\a}{(f_0f_1)^s}$
to a $(p-2)$-form $\Res(\om)$ with
\begin{equation}\label{eq:degBoundRes}
\deg\Res(\om)\le (2d_0-d_1+1)^{2s-1}\deg\a.
\end{equation}
\end{theorem}

We fix the notations and assumptions of this theorem for the rest of the section.
For its proof we will need the completion $\hat{A}=\hat{A}_{I}$ of~$A$ with respect to~$I$.
Recall from~\S\ref{ss:completions} that $\hat{A}\simeq A[[T_0,T_1]]/(T_0-f_0,T_1-f_1)$, so its
elements are power series in $f_0,f_1$.
Note, however, that these power series are not unique. E.g., $f_0\in A$ can be
represented by the constant power series $f_0$ or by $T_0$. The
crucial result for us is a lemma of Grothendieck stating that there is an
algebra isomorphism $B[[T_0,T_1]]\to\hat{A}$ (cf.~\cite[Lemma II,1.2]{har:75}),
which establishes a unique power series representation for the completion. We
need to construct this isomorphism explicitly in order to bound degrees.
The technical construction is in the following statement, which is a consequence
of the fact that~$B$ is a \textit{formally smooth} $k$-algebra~\cite[Definition 19.3.1]{gro:64}.

For a tuple $x=(x_1,\ldots,x_n)$ over an affine algebra $C$ we write
$\deg(x):=\max_j\deg(x_j)$ and use an analogous notation for the order. If
$\psi\colon C\to D$ is a homomorphism, we write $\psi(x):=(\psi(x_1),\ldots,\psi(x_n))$.
For $x\in C$ we denote by~$\overline{x}$ its image in any factor algebra of $C$.

\begin{lemma}\label{lem:groth}
Let $N\in\N$ and $\psi\colon B\to A/I^N$ be an algebra homomorphism that
lifts the identity $B\to B$, i.e., the composition $B\to A/I^N\twoheadrightarrow B$
is the identity. Then $\psi$ can be lifted to an algebra homomorphism
$\tilde{\psi}\colon B\to A/I^{N+1}$, i.e., the diagram
$$
\xymatrix{
 & A/I^{N+1} \ar@{->>}[d]^\pi \\
B \ar[r]^\psi\ar@{-->}[ru]^{\tilde{\psi}} &  A/I^N   \\
}
$$
commutes.
\end{lemma}
\begin{proof}
Since the $k$-algebra $B$ is generated by $\overline{X}_0,\ldots,\overline{X}_n$,
it is sufficient to define $\tilde{\psi}$ on these elements.
Choose $Y_0,\ldots,Y_n\in A$ such that $\psi(\overline{X}_i)=\overline{Y}_i$ in
$A/I^N$for all $0\le i\le n$. Our aim is to define
\begin{equation}\label{eq:defPsi}
\tilde{\psi}(\overline{X}_i):=\overline{Y_i+\sum_{\mu+\nu=N}a^{(i)}_{\mu\nu}f_0^\mu f_1^\nu} \quad\text{for} \quad 0\le i\le n,
\end{equation}
with suitably chosen $a^{(i)}_{\mu\nu}\in A$. Then it is clear that $\pi\circ\tilde{\psi}=\psi$.
It remains to show that one can define $\tilde{\psi}$ unambigously by~\eqref{eq:defPsi}.
This means that we have to find~$a^{(i)}_{\mu\nu}$ such that $f_0,f_1$ are mapped
to zero in $A/I^{N+1}$. Set $Y:=(Y_0,\ldots,Y_n)$, $a_{\mu\nu}:=(a^{(0)}_{\mu\nu},\ldots,a^{(n)}_{\mu\nu})$,
and look at the first condition
\begin{equation}\label{eq:cond}
f_0\mapsto \overline{f_0(Y+\sum_{\mu+\nu=N}a_{\mu\nu}f_0^\mu f_1^\nu)}=0
\quad \text{in}\quad A/I^{N+1}.
\end{equation}
By the Taylor formula we have
$$
f_0\bigg(Y+\sum_{\mu+\nu=N}a_{\mu\nu}f_0^\mu f_1^\nu\bigg)
\equiv f_0(Y)+\sum_{i=0}^n\partder{f}{X_i}(Y)\sum_{\mu+\nu=N}a^{(i)}_{\mu\nu}f_0^\mu f_1^\nu
\pmod{I^{N+1}}.
$$
Since $\overline{f_0(Y)}=\psi(\overline{f_0})=0$ in $A/I^{N}$ and
$I^{N}=(f_0^N,f_0^{N-1}f_1,\ldots,f_0f_1^{N-1},f_1^N)$, there exist $p_{\mu\nu}\in A$, $\mu+\nu=N$,
such that $f_0(Y)=\sum_{\mu+\nu=N}p_{\mu\nu}f_0^\mu f_1^\nu$ in~$A$. Furthermore, since
$\overline{Y}_i=\overline{X}_i$ in $B$, condition~\eqref{eq:cond} is satisfied if
\begin{equation}\label{eq:cond1LGS}
p_{\mu\nu}+\sum_{i=0}^n\partder{f_0}{X_i}a^{(i)}_{\mu\nu}\equiv 0\pmod{I}.
\end{equation}
Similarly, there exist $q_{\mu\nu}\in A$, $\mu+\nu=N$, such that $f_1(Y)=\sum_{\mu+\nu=N}q_{\mu\nu}f_0^\mu f_1^\nu$
in~$A$, and $f_1$ is mapped to zero if
\begin{equation}\label{eq:cond2LGS}
q_{\mu\nu}+\sum_{i=0}^n\partder{f_1}{X_i}a^{(i)}_{\mu\nu}\equiv 0\pmod{I}.
\end{equation}
In order to get very efficient degree bounds, we use the special form of the defining equations
$f_0,f_1$. In particular, recall that there exists $h\in A$ such that $\partder{f_0}{X_0}=g=h\partder{f_1}{X_n}$.
Also, note that $\partder{f_1}{X_0}=0$.
This allows us to solve the linear system of equations over $B$ consisting of~\eqref{eq:cond1LGS}
and~\eqref{eq:cond2LGS} as follows:
\begin{align}\label{eq:solLGS}
a_{\mu\nu}^{(n)} & := -q_{\mu\nu} h X_0,\nonumber \\
a_{\mu\nu}^{(i)} & := 0\quad\text{for}\quad 1\le i<n,\\
a_{\mu\nu}^{(0)} & := -X_0\big(p_{\mu\nu}+\partder{f_0}{X_n}a_{\mu\nu}^{(n)}\big).\nonumber
\end{align}
We check that these settings actually solve the system:
\begin{equation}\label{eq:check2LGS}
q_{\mu\nu}+\sum_{i=0}^n\partder{f_1}{X_i}a^{(i)}_{\mu\nu} \equiv q_{\mu\nu}-\partder{f_1}{X_n}q_{\mu\nu}h X_0
\equiv -q_{\mu\nu}f_0\equiv 0\pmod{I}.
\end{equation}
Moreover, we have
\begin{align}\label{eq:check1LGS}
p_{\mu\nu}+\sum_{i=0}^n\partder{f_0}{X_i}a^{(i)}_{\mu\nu} & \equiv p_{\mu\nu}-\partder{f_0}{X_0}X_0
\big(p_{\mu\nu}+\partder{f_0}{X_n}a_{\mu\nu}^{(n)}\big)+\partder{f_0}{X_n}a_{\mu\nu}^{(n)}\nonumber\\
 & \equiv -\big(p_{\mu\nu}+\partder{f_0}{X_n}a_{\mu\nu}^{(n)}\big) f_0
\equiv 0\pmod{I},
\end{align}
which concludes the proof.
\end{proof}

\begin{corollary}\label{cor:psi}
There exists an embedding $\psi\colon B\hookrightarrow\hat{A}$ such that
$\psi(\overline{X}_i)=\sum_{\mu,\nu=0}^\infty a_{\mu\nu}^{(i)}f_0^\mu f_1^\nu$, where $a_{\mu\nu}^{(i)}\in A$
with
$$
\deg a_{\mu\nu}^{(i)}\le 2d_0-d_1+1\quad\text{for all}\quad \mu,\nu\in\N,\ 0\le i\le n.
$$
\end{corollary}
\begin{proof}
We start with $\psi_1:=\id_B$ and apply Lemma~\ref{lem:groth} successively to construct the
homomorphisms $\psi_N\colon B\to A/I^N$, $N\in\N$. Together they define a homomorphism
$\psi\colon B\to\hat{A}$, which is clearly injective.
To prove the degree bound, denote by $Y^{(N)}=(Y_0^{(N)},\ldots,Y_n^{(N)})\in A^{n+1}$ representatives
of $\psi_N(\overline{X}_i)=\overline{Y_i^{(N)}}$ in $A/I^N$. Furthermore, let
$p_{\mu\nu}^{(N)}, q_{\mu\nu}^{(N)}\in A$ with
$$
f_0(Y^{(N)}) = \sum_{\mu+\nu=N}p_{\mu\nu}^{(N)}f_0^\mu f_1^\nu,\qquad
f_1(Y^{(N)}) = \sum_{\mu+\nu=N}q_{\mu\nu}^{(N)}f_0^\mu f_1^\nu.
$$
Then, equation~\eqref{eq:check2LGS} shows that
$$
\deg q_{\mu\nu}^{(N+1)}\le\deg q_{\mu-1,\nu}^{(N)}\quad \text{for all}\quad N\ge 1.
$$
Since $q_{01}^{(1)}=1$ and $q_{10}^{(1)}=0$, we conclude inductively that
$$
\deg q_{\mu\nu}^{(N)}\le 0\quad \text{for all}\quad N\ge 1.
$$
Moreover, by~\eqref{eq:check1LGS} we have
$$
\deg p_{\mu\nu}^{(N+1)} \le\max\{\deg p_{\mu-1,\nu}^{(N)},d_0-1+\deg h+1\}.
$$
Using $p_{01}^{(1)}=0$ and $p_{10}^{(1)}=1$, an induction proves
$$
\deg p_{\mu\nu}^{(N)}\le d_0+\deg h= 2d_0-d_1\quad \text{for all}\quad N\ge 1.
$$
The solution~\eqref{eq:solLGS} implies $\deg a_{\mu\nu}^{(n)}=d_0-d_1+1$, and hence for all $i$
\begin{equation*}
\deg a_{\mu\nu}^{(i)}\le 1+\max\{2d_0-d_1,d_0-1+d_0-d_1+1\}=2d_0-d_1+1.\qedhere
\end{equation*}
\end{proof}

For $a\in \hat{A}$ we write $\deg_{\mu\nu} a\le \d_{\mu\nu}$, if there exists a representation
$a=\sum_{\mu,\nu\ge 0} a_{\mu\nu} f_0^\mu f_1^\nu$ with $\deg a_{\mu\nu}\le \d_{\mu\nu}$ for all
$\mu,\nu\in\N$. Denote $\xi_i:=\psi(\overline{X}_i)$ and $\xi:=(\xi_0,\ldots,\xi_n)$. Then the
degree bound of the previous corollary reads in this notation
$\deg_{\mu\nu}(\xi)\le 2d_0-d_1+1=:\g$.

\begin{remark}
A straight-forward induction with respect to the degree of $p$ shows
\begin{equation}\label{eq:degNuPoly}
\deg_{\mu\nu} p(\xi)\le\g\deg p\quad\text{for all}\quad p\in A.
\end{equation}
\end{remark}

\begin{corollary}\label{cor:directSum}
For all $a=\sum_{\mu,\nu\ge 0} a_{\mu\nu} f_0^\mu f_1^\nu\in \hat{A}$ there exist unique $b\in B$ and
$c\in\hat{I}=I\hat{A}$ with $a=\psi(b)+c$. Furthermore, we have
\begin{equation}\label{eq:degBoundB}
\deg b\le\deg a_{00},
\end{equation}
and there exist $d,e\in\hat{A}$ with $c=d f_0+e f_1$ and
\begin{equation}\label{eq:degBoundDE}
\begin{split}
\deg_{\mu\nu} d & \le \max\{\deg a_{\mu+1,\nu},\g\deg a_{00}\}, \\
\deg_{\mu\nu} e & \le \max\{\deg a_{\mu,\nu+1},\g\deg a_{00}\}.
\end{split}
\end{equation}
\end{corollary}
\begin{proof}
We have the exact sequence of $\hat{A}$-modules
$$
0\longrightarrow \hat{I}\longrightarrow\hat{A}\stackrel{\pi}{\longrightarrow}
B\longrightarrow 0,
$$
which splits by the homomorphism $\psi$. For this reason, $\hat{A}\simeq B\oplus \hat{I}$, and
the existence and uniqueness of the claimed representation follows. Note that if
$a=\psi(b)+c$, then $b=\pi(a)=\pi(a_{00})$. This implies~\eqref{eq:degBoundB}.

Since $\psi$ and $\pi$ are $k$-algebra homomorphisms and $a_{00}$ is a polynomial, we
have $\psi(b)=\psi(\pi(a_{00}))=a_{00}(\xi)$, thus $c=a-a_{00}(\xi)$,
and using~\eqref{eq:degNuPoly} we conclude
$\deg_{\mu,\nu} c\le\max\{\deg a_{\mu\nu},\g\deg a_{00}\}$, which yields~\eqref{eq:degBoundDE}.
\end{proof}

Now we define the homomorphism
$$
\hat{\psi}\colon B[[T_0,T_1]]\to\hat{A},\quad\sum_{\mu,\nu\ge 0} b_{\mu\nu} T_0^\mu T_1^\nu\mapsto
\sum_{\mu,\nu\ge 0} \psi(b_{\mu\nu}) f_0^\mu f_1^\nu.
$$
\begin{lemma}\label{lem:iso}
The homomorphism $\hat{\psi}$ is an isomorphism. For $a\in A$
we have 
$$
\deg_{\mu\nu} \hat{\psi}^{-1}(a) \le \g^{\mu+\nu}\deg a\quad\text{for}\quad\mu,\nu\ge 0.
$$
\end{lemma}
\begin{proof}
We first prove that $\hat{\psi}$ is injective. It is clear that $\psi_N\colon B\to A/I^N=\hat{A}/\hat{I}^N$
is injective for all $N\ge 1$. Note also, that by construction the diagram
$$
\xymatrix{
B\ \ar@{>->}[r]^\psi \ar@{>->}[rd]^{\psi_N} & \hat{A} \ar@{->>}[d] \\
 & \hat{A}/\hat{I}^N \\
}
$$
commutes.
We show that for all $N\ge 1$ and all $b_{\mu\nu}\in B$, $\mu+\nu=N-1$, we have
\begin{equation}\label{eq:injInd}
\sum_{\mu+\nu= N-1}\psi(b_{\mu\nu})f_0^\mu f_1^\nu\equiv 0\pmod{\hat{I}^{N}}\quad\Rightarrow\quad
b_{\mu\nu}=0\text{ for all } \mu,\nu.
\end{equation}
For $N=1$ we have $0\equiv \psi(b_{00})\equiv\psi_1(b_{00})\pmod{\hat{I}}$, hence
$b_{00}=0$ by the injectivity of $\psi_1$. Now assume that~\eqref{eq:injInd} is true for some
$N\ge 1$,
and assume that $\sum_{\mu+\nu=N}\psi(b_{\mu\nu})f_0^\mu f_1^\nu\equiv 0\pmod{\hat{I}^{N+1}}$.
Reducing mod $(f_0)$ yields $\psi(b_{0,N})f_1^N \equiv 0\pmod{(f_0)}$. Since $f_1$ is a
non-zerodivisor mod $(f_0)$, we conclude $\psi(b_{0,N}) \equiv 0\pmod{(f_0)}$,
thus $0\equiv\psi(b_{0,N}) \equiv\psi_1(b_{0,N})\pmod{\hat{I}}$. Injectivity of $\psi_1$ implies
$b_{0,N}=0$. Now we write 
$$
\psi(b_{N,0})f_0^N+\psi(b_{\N-1,1})f_0^{N-1} f_1+\cdots+\psi(b_{1,N-1})f_0 f_1^{N-1}=f_0\cdot a,
$$
where $a\in\hat{A}$ has the form of the assumption in~\eqref{eq:injInd}. 
Since $f_0 a\equiv 0\pmod{\hat{I}^{N+1}}$, there exist $g_{\la\eta}\in \hat{A}$ with
$f_0 a=\sum_{\la+\eta=N+1} g_{\la\eta}f_0^\la f_1^\eta$.
Since $f_1$ is a non-zerodivisor mod $(f_0)$, we have $g_{0,N+1}\equiv 0\pmod{(f_0)}$, and since $f_0$ is a
non-zerodivisor, we infer
$a\equiv 0\pmod{\hat{I}^{N}}$. The induction hypothesis implies $b_{\mu\nu}=0$ for all
$\mu,\nu$ with $\mu+\nu=N$ and $\mu>0$, which completes the proof of~\eqref{eq:injInd}.

Now let $b=\sum_{\mu,\nu\ge 0}b_{\mu\nu}T_0^\mu T_1^\nu\in B[[T_0,T_1]]$ with $\hat{\psi}(b)=0$. We
apply~\eqref{eq:injInd} inductively to conclude that $b=0$.

To show surjectivity, let $a\in\hat{A}$ and construct a preimage $\sum_{\mu,\nu} b_{\mu\nu} T_0^\mu T_1^\nu\in B[[T_0,T_1]]$
of~$a$ under $\hat{\psi}$.
We find the $b_{\mu\nu}$ successively by
applying Corollary~\ref{cor:directSum}. Let $b_{00}\in B$ and $d_{10},d_{01}\in \hat{A}$ with
$a=\psi(b_{00})+d_{10} f_0+d_{01}f_1$. 
Assume inductively, that for some $N\ge 1$ we have constructed $b_{\mu\nu}\in B$ for $\mu+\nu<N$,
and $d_{\mu\nu}\in\hat{A}$ for $\mu+\nu=N$, such that
\begin{equation}\label{eq:successiveCoeffInd}
a=\sum_{\mu+\nu<N}\psi(b_{\mu\nu})f_0^\mu f_1^\nu+\sum_{\mu+\nu=N} d_{\mu\nu}f_0^\mu f_1^\nu.
\end{equation}

Then, for all $\mu,\nu$ with $\mu+\nu=N$ we obtain from Corollary~\ref{cor:directSum} elements
$b_{\mu\nu}\in B$ and $d_{\mu\nu}^0,d_{\mu\nu}^1\in\hat{A}$ such that
$$
d_{\mu\nu}=\psi(b_{\mu\nu})+d_{\mu\nu}^0f_0+d_{\mu\nu}^1f_1.
$$
Plugging into~\eqref{eq:successiveCoeffInd} yields
$$
a=\sum_{\mu+\nu\le N}\psi(b_{\mu\nu})f_0^\mu f_1^\nu+\sum_{\mu+\nu=N}(d_{\mu\nu}^0f_0+d_{\mu\nu}^1f_1)f_0^\mu f_1^\nu,
$$
which is of the form~\eqref{eq:successiveCoeffInd} for $N+1$ and hence completes the induction.
We have $\sum_{\mu,\nu\ge 0} \psi(b_{\mu\nu}) f_0^\mu f_1^\nu=a$, since this equality holds modulo
$\hat{I}^N$ for all $N\ge 1$.

Now assume that $a\in A$. Then we claim
\begin{equation}\label{eq:degD}
\deg_{\la\eta}(d_{\mu\nu})\le\g^{\mu+\nu}\deg a\quad\text{for all}\quad \mu+\nu=N,
\end{equation}
which for $N=1$ follows directly from~\eqref{eq:degBoundDE}.
Assuming~\eqref{eq:degD} for some $N\ge 1$, \eqref{eq:degBoundDE} implies
$$
\deg_{\la\eta}d_{\mu\nu}^0\le\max\{\deg_{\la+1,\eta} d_{\mu\nu},\g\deg_{00} d_{\mu\nu}\}\le\g^{\mu+\nu+1}\deg a,
$$
and $d_{\mu\nu}^0$ contributes to $d_{\mu+1,\nu}$. A similar estimate holds for $d_{\mu\nu}^1$, which
completes the proof of~\eqref{eq:degD}.

Finally, \eqref{eq:degD} and~\eqref{eq:degBoundB} yield
\begin{equation*}
\deg(b_{\mu\nu})\le\deg_{00} d_{\mu\nu}\le\g^{\mu+\nu}\deg a\quad\text{for all}\quad \mu,\nu\in\N.\qedhere
\end{equation*}
\end{proof}

\begin{proof}[Proof of Theorem~\ref{thm:effGysin}]
We prove the Theorem from scratch by constructing the residue map and checking that it is an
isomorphism.

W.l.o.g.~we can assume $p\ge 2$. As stated in the Theorem, we identify
$$
\Hdr^{p+1}(W)\simeq H^{p}\left(\frac{\Om_{A_{f_0f_1}}^\bullet}{\Om_{A_{f_0}}^\bullet+\Om_{A_{f_1}}^\bullet}\right).
$$
Consider the map
$$
\la\colon \Hdr^{p-2}(Z)\to H^{p}\left(\frac{\Om_{A_{f_0f_1}}^\bullet}{\Om_{A_{f_0}}^\bullet+\Om_{A_{f_1}}^\bullet}\right), 
\quad [\overline{\om}]\mapsto\left[\overline{\frac{\ud f_0}{f_0}\wedge\frac{\ud f_1}{f_1}\wedge\om}\right],
$$
where $\om\in\Om_A^{p-2}$, $\overline{\om}$ denotes its image in $\Om_B^{p-2}$, and 
$[\overline{\om}]$ the cohomology class in $\Hdr^{p-2}(Z)$ it represents. A similar notation is
used on the right hand side.

We first show that the map $\la$ is well-defined.
If $\om\in\Om_A^{p-2}$ represents the zero cohomology class in $\Hdr^{p-2}(Z)$, then
$\overline{\om}=\ud\overline{\eta}$ for some $\eta\in\Om_A^{p-3}$. This means that $\om-\ud\eta$ is
contained in the differential graded ideal generated by $I$ and $\ud I$. Using the formula
$\ud f_i\wedge\a=\ud(f_i\a)-f_i\ud \a$,
we can assume that there exist $\a,\b\in\Om_A^{p-2}$ with $\om-\ud\eta=f_0\a+f_1\b$. Then 
\begin{align*}
\frac{\ud f_0}{f_0}\wedge\frac{\ud f_1}{f_1}\wedge\om & =\ud f_0\wedge\frac{\ud f_1}{f_1}\wedge\a+
\frac{\ud f_0}{f_0}\wedge\ud f_1\wedge\b+
\frac{\ud f_0}{f_0}\wedge\frac{\ud f_1}{f_1}\wedge\ud\eta \\
 & \equiv\ud\left(\frac{\ud f_0}{f_0}\wedge\frac{\ud f_1}{f_1}\wedge\eta\right)
\bmod\left(\Om_{A_{f_0}}^\bullet+\Om_{A_{f_1}}^\bullet\right),
\end{align*}
so it maps to zero in the cohomology on the right hand side.
Furthermore, since $\ud f_i/f_i$ is exact, one easily checks that $\la$ sends closed (exact) forms
to closed (exact) ones.

The residue map will be the inverse of $\la$. To construct it, note that the isomorphism of
Lemma~\ref{lem:iso} induces a homomorphism $\Om_{A_{f_0f_1}}^\bullet\hookrightarrow\hat{\Om}_{\hat{A}_{f_0f_1}}^\bullet
\stackrel{\simeq}{\longrightarrow}\Om_{B[[T_0,T_1]]_{T_0T_1}}^\bullet$
and as a result an isomorphism
$$
\th\colon\frac{\Om_{A_{f_0f_1}}^\bullet}{\Om_{A_{f_0}}^\bullet+\Om_{A_{f_1}}^\bullet}\simeq
\frac{\hat{\Om}_{\hat{A}_{f_0f_1}}^\bullet}{\hat{\Om}_{\hat{A}_{f_0}}^\bullet+\hat{\Om}_{\hat{A}_{f_1}}^\bullet}
\stackrel{\simeq}{\longrightarrow}
\frac{\hat{\Om}_{B[[T_0,T_1]]_{T_0T_1}}^\bullet}{\hat{\Om}_{B[[T_0,T_1]]_{T_0}}^\bullet+\hat{\Om}_{B[[T_0,T_1]]_{T_1}}^\bullet}.
$$
Define
$$
\Res\colon H^p\left(\frac{\Om_{A_{f_0f_1}}^\bullet}{\Om_{A_{f_0}}^\bullet+\Om_{A_{f_1}}^\bullet}\right)
\longrightarrow\Hdr^{p-2}(Z)
$$
as follows. For a form $\om\in\Om_{A_{f_0f_1}}^p$ write
\begin{equation}\label{eq:formExpansion}
\th(\overline{\om})=\sum_{\mu,\nu\ge 1}(\a_{\mu\nu}+\b_{\mu\nu}\wedge\ud T_0+\g_{\mu\nu}\wedge\ud T_1+
\d_{\mu\nu}\wedge\ud T_0\wedge\ud T_1)T_0^{-\mu} T_1^{-\nu}
\end{equation}
with $\a_{\mu\nu},\b_{\mu\nu},\g_{\mu\nu},\d_{\mu\nu}\in\Om_B^\bullet$, only finitely many non-zero, 
and map $\Res([\overline{\om}]):=[\d_{1,1}]$. 
For the proof that this maps closed forms to closed forms,
see below.

To prove that this defines the inverse of $\la$, it suffices to show $\Res\circ\la=\id$ and
$\la\circ\Res=\id$ (in particular, this implies well-definedness). First, for
$[\overline{\om}]\in \Hdr^{p-2}(Z)$ we have
$$
\Res\circ\la([\overline{\om}])=\Res\left(\left[\overline{\frac{\ud f_0}{f_0}\wedge\frac{\ud f_1}{f_1}\wedge\om}\right]\right)=[\overline{\om}],
$$
since
$\th(\overline{\frac{\ud f_0}{f_0}\wedge\frac{\ud f_1}{f_1}\wedge\om})=\overline{\frac{\ud T_0}{T_0}\wedge\frac{\ud T_1}{T_1}\wedge\overline{\om}}$.

On the other hand, let $\om\in\Om_{A_{f_0f_1}}^p$ be a form with $\ud\om\in\Om_{A_{f_0}}^p+\Om_{A_{f_1}}^p$.
Writing $\th(\overline{\om})$ in the form~\eqref{eq:formExpansion} and differentiating yields
\begin{align*}
\ud\th(\overline{\om}) = & \sum_{\mu,\nu\ge 1}\Big((\ud\a_{\mu\nu}+\ud\b_{\mu\nu}\wedge\ud T_0+\ud\g_{\mu\nu}\wedge\ud T_1+
\ud\d_{\mu\nu}\wedge\ud T_0\wedge\ud T_1)T_0^{-\mu} T_1^{-\nu} \\
 & -(-1)^p\mu(\a_{\mu\nu}+\g_{\mu\nu}\wedge\ud T_1)T_0^{-\mu-1} T_1^{-\nu}\ud T_0 \\
 & -(-1)^p\nu(\a_{\mu\nu}+\b_{\mu\nu}\wedge\ud T_0)T_0^{-\mu} T_1^{-\nu-1}\ud T_1\Big) \\
 = & \sum_{\mu,\nu\ge 1}T_0^{-\mu} T_1^{-\nu}\Big(\ud\a_{\mu\nu}+(\ud\b_{\mu\nu}+(-1)^{p+1}(\mu-1)\a_{\mu-1,\nu})\wedge\ud T_0 \\
 & +
(\ud\g_{\mu\nu}+(-1)^{p+1}(\nu-1)\a_{\mu,\nu-1})\wedge\ud T_1 \\
 &  + (\ud\d_{\mu\nu} +(-1)^p(\mu-1)\g_{\mu-1,\nu}+(-1)^{p+1}(\nu-1)\b_{\mu,\nu-1})\wedge\ud T_0\wedge\ud T_1\Big).
\end{align*}
Since in all terms of this expression both $T_0$ and $T_1$ have negative exponents and are
contained in $\hat{\Om}_{B[[T_0,T_1]]_{T_0}}^p+\hat{\Om}_{B[[T_0,T_1]]_{T_1}}^p$, they must be zero in
$\hat{\Om}_{B[[T_0,T_1]]_{T_0T_1}}^p$. Among others, this implies the relations
\begin{gather}
\ud\b_{\mu\nu}+(-1)^{p+1}(\mu-1)\a_{\mu-1,\nu} = 0,\label{eq:closednessCond2} \\
\ud\d_{\mu\nu} +(-1)^p(\mu-1)\g_{\mu-1,\nu}+(-1)^{p+1}(\nu-1)\b_{\mu,\nu-1} = 0\label{eq:closednessCond4}
\end{gather}
for all $\mu,\nu\ge 1$. In particular, \eqref{eq:closednessCond4} shows that $\d_{1,1}$ is closed.
We have to show that all other terms of $\om$ can be integrated. To do so, define
$$
\eta:=\sum_{\mu\ge 2}\sum_{\nu\ge 1}(-1)^{p}\frac{T_0^{1-\mu}}{1-\mu}T_1^{-\nu}(-\b_{\mu\nu}+\d_{\mu\nu}\wedge\ud T_1)
+\sum_{\nu\ge 2}(-1)^{p+1}T_0^{-1}\frac{T_1^{1-\nu}}{1-\nu}\d_{1,\nu}\wedge\ud T_0.
$$
We check
\begin{align*}
\ud\eta = & \sum_{\mu\ge 2}\sum_{\nu\ge 1}(-1)^p\Big((T_0^{-\mu}T_1^{-\nu}\ud T_0-\frac{T_0^{1-\mu}}{1-\mu}\nu T_1^{-\nu-1}\ud T_1)\wedge(-\b_{\mu\nu}+\d_{\mu\nu}\wedge\ud T_1) \\
 & +\frac{T_0^{1-\mu}}{1-\mu}T_1^{-\nu}(-\ud\b_{\mu\nu}+\ud\d_{\mu\nu}\wedge\ud T_1)\Big) \\
 & +\sum_{\nu\ge 2}(-1)^{p+1}\Big((-T_0^{-2}\frac{T_1^{1-\nu}}{1-\nu}\ud T_0+T_0^{-1}T_1^{-\nu}\ud T_1)
\wedge\d_{1,\nu}\wedge\ud T_0 \\
 & +T_0^{-1}\frac{T_1^{1-\nu}}{1-\nu}\ud\d_{1,\nu}\wedge\ud T_0\Big) \\
 = & \sum_{\mu\ge 2}\sum_{\nu\ge 1}\Big((-1)^{p+1}\frac{T_0^{1-\mu}}{1-\mu}T_1^{-\nu}\ud\b_{\mu\nu}
+T_0^{-\mu}T_1^{-\nu}\b_{\mu\nu}\wedge\ud T_0 \\
 & +\frac{T_0^{1-\mu}}{1-\mu}(-\nu T_1^{-\nu-1}\b_{\mu\nu}+(-1)^pT_1^{-\nu}\ud\d_{\mu\nu})\wedge\ud T_1
 +T_0^{-\mu}T_1^{-\nu}\d_{\mu\nu}\ud T_0\wedge \ud T_1\Big) \\
 & +\sum_{\nu\ge 2}T_0^{-1}\Big(T_1^{-\nu}\d_{1,\nu}\wedge\ud T_0\wedge\ud T_1+(-1)^{p+1}\frac{T_1^{1-\nu}}{1-\nu}\ud\d_{1,\nu}\wedge\ud T_0\Big).
\end{align*}
Using \eqref{eq:closednessCond2} and \eqref{eq:closednessCond4} for $\mu=1$ we obtain
\begin{align*}
\ud\eta = & \sum_{\mu \ge 2}\sum_{\nu\ge 1}\Big(T_0^{1-\mu}T_1^{-\nu}\a_{\mu-1,\nu}
+T_0^{-\mu}T_1^{-\nu}\b_{\mu\nu}\wedge\ud T_0 \\
 & +\frac{T_0^{1-\mu}}{1-\mu}T_1^{-\nu}(-(\nu-1)\b_{\mu,\nu-1}+(-1)^p\ud\d_{\mu\nu})\wedge\ud T_1\Big) \\
 &  +\sum_{\substack{\mu,\nu\ge 1 \\ (\mu,\nu)\ne(1,1)}}T_0^{-\mu}T_1^{-\nu}\d_{\mu\nu}\ud T_0\wedge \ud T_1 
+\sum_{\nu\ge 2}T_0^{-1}T_1^{1-\nu}\b_{1,\nu-1}\wedge\ud T_0 \\
\end{align*}
\begin{align*}
 \stackrel{\eqref{eq:closednessCond4}}{=} & \sum_{\mu\ge 1}\sum_{\nu=1}^sT_0^{-\mu}T_1^{-\nu}(\a_{\mu,\nu}
+\b_{\mu\nu}\wedge\ud T_0) +\sum_{\mu\ge 2}\sum_{\nu\ge 1}T_0^{1-\mu}T_1^{-\nu}\g_{\mu-1,\nu}\wedge\ud T_1 \\
 &  +\sum_{\substack{\mu,\nu\ge 1 \\ (\mu,\nu)\ne(1,1)}}T_0^{-\mu}T_1^{-\nu}\d_{\mu\nu}\ud T_0\wedge \ud T_1 \\
 = &\; \th(\overline{\om})-T_0^{-1}T_1^{-1}\d_{1,1}\ud T_0\wedge \ud T_1,
\end{align*}
thus $\la\circ\Res([\overline{\om}])=[\overline{\om}]$.

In order to prove~\eqref{eq:degBoundRes}, we have to bound the degree of $\d_{1,1}$ in~\eqref{eq:formExpansion}.
Note that by linearity it suffices to consider terms of the form
$$
\om=\frac{a}{(f_0f_1)^s}\ud X_{i_1}\wedge\cdots\wedge\ud X_{i_p},\quad a\in A,\ 0\le i_1<\cdots<i_p\le n,\ s\ge 1.
$$
By Lemma~\ref{lem:iso} we have
\begin{align*}
b:= & \hat{\psi}^{-1}(a)=\sum_{\mu,\nu\ge 0} b_{\mu\nu} T_0^\mu T_1^\nu,\\
\Xi_i:= & \hat{\psi}^{-1}(X_i)=\sum_{\mu,\nu\ge 0} b_{\mu\nu}^{(i)} T_0^\mu T_1^\nu\in B[[T_0,T_1]],
\end{align*}
where
\begin{equation}\label{eq:boundCoeff}
\deg b_{\mu\nu}\le \g^{\mu+\nu}\deg a,\quad \deg b_{\mu\nu}^{(i)}\le \g^{\mu+\nu}.
\end{equation}
It follows
$$
\th(\overline{\om})=\frac{b}{(T_0T_1)^s}\ud \Xi_{i_1}\wedge\cdots\wedge\ud \Xi_{i_p}.
$$
Moreover, we have
\begin{align*}
\ud \Xi_i & = \sum_{\mu,\nu\ge 0} (\ud b_{\mu\nu}^{(i)} T_0^\mu T_1^\nu+b_{\mu\nu}^{(i)} \mu T_0^{\mu-1} T_1^\nu\ud T_0
+b_{\mu\nu}^{(i)} \nu T_0^{\mu} T_1^{\nu-1}\ud T_1) \\
 & = \sum_{\mu,\nu\ge 0} (\ud b_{\mu\nu}^{(i)}+(\mu+1)b_{\mu+1,\nu}^{(i)}\ud T_0
+(\nu+1)b_{\mu,\nu+1}^{(i)}\ud T_1)T_0^\mu T_1^\nu.
\end{align*}
The terms of $\th(\overline{\om})$ involving $\ud T_0\wedge\ud T_1$ are of the form
\begin{multline*}
\pm(\mu_1+1)(\nu_2+1) b_{\mu\nu} b_{\mu_1+1,\nu_1}^{(i)} b_{\mu_2,\nu_2+1}^{(j)} T_0^{\mu+\mu_1+\cdots+\mu_p-s}
T_1^{\nu+\nu_1+\cdots+\nu_p-s} \\
\cdot\ud T_0\wedge\ud T_1\wedge\ud_B b_{\mu_3,\nu_3}^{(j_1)}\wedge\cdots\wedge\ud_B b_{\mu_p,\nu_p}^{(j_{p-2})}
\end{multline*}
with some $0\le i,j,j_1,\ldots,j_{p-2}\le n$ and $\mu,\nu,\mu_1,\nu_1,\ldots,\mu_p,\nu_p\ge 0$.
To get the coefficient $\d_{1,1}$ of $\ud T_0/T_0\wedge\ud T_1/T_1$, we have to consider the case
$\mu+\mu_1+\cdots+\mu_p=s-1$ and $\nu+\nu_1+\cdots+\nu_p=s-1$.
Using that $\ud_B$ is of degree $0$ together with the estimate~\eqref{eq:boundCoeff}, it follows
that $\d_{1,1}$ is of degree
\begin{align*}
 & \le \deg b_{\mu\nu} +\deg b_{\mu_1+1,\nu_1}^{(i)} +\deg b_{\mu_2,\nu_2+1}^{(j)}
+\deg b_{\mu_3,\nu_3}^{(j_1)} + \cdots+ \deg b_{\mu_p,\nu_p}^{(j_{p-2})} \\
 & \le \g^{\mu+\nu}\deg a+\g^{\mu_1+\nu_1+1}+\g^{\mu_2+\nu_2+1}+\g^{\mu_3+\nu_3}+\cdots++\g^{\mu_p+\nu_p} \\
 & \le \g^{2s-1}(\deg a+p).\qedhere
\end{align*}
\end{proof}

\section{Proof of the Main Theorem}
The effective Gysin sequence yields degree and order bounds for the de Rham cohomology of a smooth
hypersurface.
\begin{theorem}\label{thm:effDeRhamHyper}
Let $f,g\in R=k[X_1,\ldots,X_n]$ with $d:=\deg f$ and $d':=\deg g$, such that $f$ is irreducible
and $\partder{f}{X_n}|g$, and consider the smooth hypersurface
$V:=\mZ(f)\setminus\mZ(g)\subseteq\A^n$. Then we have
$$
(p+2)(d+d'+2)(2d'-d+3)^{2p+3}\cdot (1,1)\in B(\Hdr^{p}(V)).
$$
for all $p\in\N$.
\end{theorem}
\begin{proof}
Putting $f_0:=gX_0-1$, $f_1:=f$, we have the isomorphism
$$
V\stackrel{\simeq}{\longrightarrow} Z:=\mZ(f_0,f_1)\subseteq\A^{n+1},\ x\mapsto (1/g(x),x),
$$
and the pull-back of differential forms shows
$$
\deg(\Hdr^p(Z))\cdot (1,1)\in B(\Hdr^p(V)).
$$
Thus, we have reduced to the setting of Theorem~\ref{thm:effGysin}. Note that $d_0:=\deg f_0=d'+1$
and $d_1:=\deg f_1=d$. By Corollary~\ref{cor:boundLocalCoh} we have
$$
p\cdot (1,1)\in B\left( H^{p}\left(\frac{\Om_{A_{f_0f_1}}^\bullet}{\Om_{A_{f_0}}^\bullet+\Om_{A_{f_1}}^\bullet}\right)\right).
$$
Theorem~\ref{thm:effGysin} implies
$$
\deg(\Hdr^{p-2}(Z))\le (2d_0-d_1+1)^{2p-1}p(d_0+d_1+1)=p(d+d'+2)(2d'-d+3)^{2p-1},
$$
which implies the claim.
\end{proof}

\begin{proof}[Proof of Theorem~\ref{thm:main}]
Let $X\subseteq\A^n$ be a smooth closed variety. 
The zeroth cohomology is treated in \S\ref{se:zerothCoh},
and the case $D=1$ is trivial, so we assume $n>m\ge 1$ and $D\ge 2$.

First assume that $X$ is irreducible. Then, by Corollary~\ref{cor:biratHyperCover}
and Lemma~\ref{lem:biratHyper}
we can write $X=\bigcup_{i=0}^m U_i$, where $U_i=X\setminus\mZ(g_i)$ is isomorphic
to $V_i=Y_i\setminus\mZ(g_i)\subseteq\A^{m+1}$. Furthermore, in suitable
coordinates $X_1,\ldots,X_n$ of $\A^n$, the isomorphism is given by the projection
$\pi\colon U_i\to V_i$ onto the first $m+1$ coordinates, and we have $Y_i=\mZ(f_i)$,
where $f_i\in k[X_1,\ldots,X_{m+1}]$ is irreducible and monic in $X_{m+1}$,
$g_i=\partder{f_i}{X_{m+1}}$,
and $\deg f_i\le D$. 
Now fix $1\le \ell\le m$. For $p,q\in\N$ with $p+q=\ell$ consider a multi-index
$0\le i_0<\cdots<i_q\le m$.
Set $i:=i_0$, $U':=U_{i_0\cdots i_q}$, and $V':=\pi(U')$.
Note that $U'=U_i\setminus\mZ(h)$ with $h:=g_{i_1}\cdots g_{i_q}$. From
Corollary~\ref{cor:transHyperPatch} we obtain $H\in k[Y_i]$ with
$\deg H\le D\deg h$ such that $k[V']=k[Y_i]_{g_iH}$. We can assume $d:=\deg f_i\ge 2$,
and with $g:=g_i H$ we have $d-1\le d':=\deg g\le d-1+\deg H\le (D-1)(qD+1)\le \ell D^2+D-2$.
Theorem~\ref{thm:effDeRhamHyper} implies
$$
(p+2)(d+d'+2)(2d'-d+3)^{2p+3}\cdot (1,1)\in B(\Hdr^{p}(V_i)).
$$
Moreover,
\begin{align*}
(p+2) & (d+d'+2)(2d'-d+3)^{2p+3} \\
  & \le (p+2)\big(D+(D-1)(qD+1)\big)\big(2(D-1)(qD+1)+1\big)^{2p+3} \\
  & \le (p+2)\big(\ell D^2+2D\big)\big(2D(q(D-1)+1)\big)^{2p+3} \\
  & \le 6\ell^2 D^2(2\ell D^2)^{2\ell+3} = 3\cdot 2^{2\ell+4}\ell^{2\ell+5} D^{4\ell+8}.
\end{align*}
Using Corollary~\ref{cor:transHyperPatch} we conclude
\begin{equation}\label{eq:boundSD}
(s,d):=3\cdot 2^{2\ell+4}\ell^{2\ell+5} D^{4\ell+8}\cdot (\ell D,1)\in B(\Hdr^p(U_{i_0\cdots i_q})),
\end{equation}
Since this holds for all $p+q=\ell$ and all $i_0<\cdots <i_q$,
Theorem~\ref{thm:effHyperCoh} shows that
\begin{align*}
\deg \Hdr^\ell(X) & \le d+2D(\ell+1)(s+\ell)^m (D-1)^m \\
 & \le 2D(\ell+1)(\ell D d+\ell)^m D^m \\
 & \le 4\ell^{m+1} (D d+1)^m D^{m+1} \\
 & \stackrel{\eqref{eq:boundSD}}{\le} 4\ell^{m+1} (2^{2\ell+6}\ell^{2\ell+5} D^{4\ell+9})^m D^{m+1} \\
 & = 2^{2\ell m+6m+2}\ell^{2\ell m+6m+1} D^{4\ell m+10m+1}.
\end{align*}
Finally, for reducible $X$, Corollary~\ref{cor:redIrred} implies the claim.
\end{proof}

\subsection*{Acknowledgements}
The author thanks Sergei Yakovenko for asking the question addressed in this
paper and bringing to his attention the solution of the infinitesimal Hilbert
16th problem~\cite{bny:10}.
He also thanks Cornelia Rottner for finding some inaccuracies in a previous version. Example~\ref{ex:degLoc} is due to her.
He is also grateful to the Hausdorff Center for Mathematics, Bonn, as well as the Lucerne University of
Applied Sciences and Arts for their kind support. 


\end{document}